\documentclass[11pt, reqno]{amsart}
\usepackage[utf8]{inputenc}
\usepackage{amsmath, amssymb, amsfonts, latexsym, amsthm}
\usepackage{calligra}
\usepackage{tcolorbox}

\newtheorem{thm}{Theorem}[section]
\newtheorem{cor}[thm]{Corollary}

\newtheorem{prop}[thm]{Proposition}
\newtheorem{rem}[thm]{Remark}
\newtheorem{obs}[thm]{Observation}
\newtheorem{defn}[thm]{Definition}
\newtheorem{exam}[thm]{Example}
\numberwithin{equation}{section}

\newcommand{\mbb}{\mathbb}

\newcommand{\z}{\zeta}
\newcommand{\pa}{\partial}
\newcommand{\ov}{\overline}

\newcommand{\Om}{\Omega}

\begin{document}
	\title[Weighted kernel functions on planar domains]{Weighted kernel functions on planar domains}
	\author{Aakanksha Jain, Kaushal Verma} 
	\address{AJ : Department of Mathematics, Indian Institute of Science Education and Research, Pune 411008, India}
    \email{aakanksha.jain@iiserpune.ac.in, aakankshaj@alum.iisc.ac.in}
	
	\address{KV: Department of Mathematics, Indian Institute of Science, Bengaluru-560012, India}
	\email{kverma@iisc.ac.in}
	
\begin{abstract}
We study the variation of weighted Szeg\H{o} and Garabedian kernels on planar domains with respect to the weight and establish a Ramadanov-type convergence theorem. As a consequence, we obtain detailed information on the behaviour of the zeros of the weighted Szeg\H{o} and Garabedian kernels for weights close to the constant function $1$ on the boundary. We also study the weighted Ahlfors map and strengthen known results concerning its boundary behaviour. Explicit examples of the weighted kernels are presented for certain classes of weights. We strengthen Nehari’s construction of the weighted Szeg\H{o} and Garabedian kernels by showing that the meromorphic function associated with the weight can be chosen holomorphic. We further identify his auxiliary kernels with reproducing kernels of finite-codimensional Hardy subspaces, and derive further representation formulas and consequences. Finally, we discuss the weighted Carathéodory metric, and describe relations between the weighted Szeg\H{o} and Garabedian kernels with certain classical kernel functions.
\end{abstract}
	
\maketitle

\section{Introduction}

\noindent Some aspects of weighted Szeg\H{o} kernels on a finitely connected planar domain were considered in \cite{JV}. We continue this study by strengthening some of the results obtained therein on the variation of the weighted Szeg\H{o} kernel and the weighted Garabedian kernel as functions of the weight and noting some consequences about the zeros of the weighted Szeg\H{o} kernel. This has direct consequences about the zeros of the weighted Garabedian kernel. Continuing along this path leads to the consideration of a weighted Ahlfors map (cf. Nehari \cite{Nehari}) associated with the given planar domain, and we are able to strengthen an observation made by Nehari in \cite{Nehari} about its boundary behaviour. We then provide explicit formulas for the weighted Szeg\H{o} kernel, the weighted Garabedian kernel, and the weighted Ahlfors map for certain classes of weights. 
We revisit and refine Nehari's construction of the weighted Szeg\H{o} and Garabedian kernels in two ways. Using a theorem of \v{C}erne, we first show that the meromorphic function associated with the boundary weight can be chosen holomorphic. Second, we identify Nehari's auxiliary Szeg\H{o} kernel with the reproducing kernel of a finite-codimensional Hardy subspace and obtain a corresponding interpretation of the auxiliary Garabedian kernel. This leads to representation formulas for the weighted kernels and several further consequences. 
Finally, we discuss the weighted Carathéodory metric, as well as expressions for the weighted Szeg\H{o} kernel and the weighted Garabedian kernel in terms of certain classical kernel functions.

\medskip

To make all this precise and fix the notation, let $\Omega \subset \mbb C$ be a smoothly bounded $n$-connected domain without isolated boundary points and $\varphi$ a positive real-valued $C^{\infty}$-smooth function on $\pa \Om$. Following \cite{Be0}, \cite{Be1}, let $L^2_{\varphi}(\partial\Omega)$ be the Hilbert space of complex-valued functions on $\partial\Omega$ that are square integrable with respect to the arc length measure $\varphi\,ds$. Here, $ds$ is the differential of arc length and is given by  $ds = \vert z'(t) \vert \; dt$ where $z = z(t)$ is a smooth parametrization of $\partial \Omega$. In terms of the complex unit tangent vector function $T(z(t)) = z'(t) / \vert z'(t) \vert$, it is seen that $dz = z'(t) \; dt = T ds$. The inner product on $L^2_{\varphi}(\partial\Omega)$ is
\begin{equation}
\langle u,v\rangle_{\varphi}=\int_{\partial\Omega}u\ov{v}\,\varphi\,ds\quad \text{for }u,v\in L^2_{\varphi}(\partial\Omega).
\end{equation}
When $\varphi \equiv 1$, this reduces to the standard inner product $\langle u, v \rangle$. Note that $L^2_{\varphi}(\partial\Omega)=L^2(\partial\Omega)$ as sets. Also, $C^{\infty}(\partial\Omega)$ is dense in $L^2_{\varphi}(\partial\Omega)$ with respect to its norm. For $u\in C^{\infty}(\partial\Omega)$, the \textit{Cauchy transform} of $u$ is 
\begin{equation}
(\mathcal{C}u)(z)=\frac{1}{2\pi i}\int_{\partial\Omega}\frac{u(\zeta)}{\zeta-z}d\zeta,\quad z\in\Omega
\end{equation}
which is holomorphic on $\Omega$. For $a \in \Omega$ and $z \in \partial \Omega$, let $C_a(z)$ be the conjugate of 
\[
\frac{1}{2 \pi i} \frac{T(z)}{z-a}.
\]
Then, $C_a$ is the \textit{Cauchy kernel} which defines the Cauchy transform $\mathcal C$ in the sense that $\mathcal C u(z) = \langle u, C_z \rangle$. The \textit{weighted Cauchy kernel} is defined to be $\varphi^{-1} C_a$ and the corresponding weighted Cauchy transform $\mathcal C_{\varphi}u$ satisfies
\[
(\mathcal C_{\varphi} u)(z) = \langle u, \varphi^{-1} C_z \rangle_{\varphi} = \langle u, C_z \rangle = (\mathcal C u)(z)
\]
and this shows that $\mathcal C_{\varphi} = \mathcal C$.

\medskip

Let $A^{\infty}(\Omega)$ denote the space of holomorphic functions on $\Omega$ that are in $C^{\infty}(\ov{\Omega})$. The Cauchy transform $\mathcal C$ maps $C^{\infty}(\partial\Omega)$ into $A^{\infty}(\Omega)$  and this allows the Cauchy transform to be viewed as an operator from $C^{\infty}(\partial\Omega)$ into itself. Let $A^{\infty}(\partial\Omega)$ denote the set of functions on $\partial\Omega$ that are the boundary values of functions in $A^{\infty}(\Omega)$. The \textit{Hardy space} $H^2(\partial\Omega)$ is defined to be the closure in $L^2(\partial\Omega)$ of $A^{\infty}(\partial\Omega)$ and members of the Hardy space are in one-to-one correspondence with the space of holomorphic functions on $\Omega$ with $L^2$ boundary values. This Hardy space can be identified in a natural way with the subspace of $L^2_{\varphi}(\partial\Omega)$, equal to the closure of $A^{\infty}(\partial\Omega)$ in that space. Thus, there is no need to define $H^2_{\varphi}(\partial\Omega)$ separately. But we shall use the notation $H^2_{\varphi}(\partial\Omega)$ whenever there is a need to emphasize that the Hardy space is endowed with the inner product $\langle \cdot,\cdot\rangle_{\varphi}$. 

\medskip

The orthogonal projection $P_{\varphi}$ from $L^2_{\varphi}(\partial\Omega)$ onto $H^2(\partial\Omega)$ is the \textit{weighted Szeg\H{o} projection}. For $a\in\Omega$, the evaluation functional $h\mapsto h(a)$ on $H^2_{\varphi}(\partial\Omega)$ is continuous since
\[
\vert h(a)\vert = \vert \langle h, C_a \rangle \vert = \vert \langle h, C_a \varphi^{-1}\rangle_{\varphi}\vert \leq \Vert C_a \varphi^{-1} \Vert_{\varphi} \Vert h\Vert_{\varphi}
\]
and hence, there exists a unique function $S_{\varphi}(\cdot,a)\in H^2(\partial\Omega)$ such that 
\begin{equation}
h(a)=\langle h, S_{\varphi}(\cdot, a)\rangle_{\varphi}
\end{equation}
for all $h\in H^2(\partial\Omega)$. The function $S_{\varphi}(\cdot,\cdot)$ is the \textit{weighted Szeg\H{o} kernel} of $\Omega$ with respect to the weight $\varphi$. 

\medskip

It can be seen that the weighted Szeg\H{o} kernel is hermitian symmetric, that is, for $z, w\in\Omega$, we have
$
S_{\varphi}(z,w)= \overline{S_{\varphi}(w,z)}.
$
Therefore, $S_{\varphi}(\cdot,\cdot)\in C^{\infty}(\Omega\times\Omega)$. 

\medskip

The weighted Szeg\H{o} kernel is the kernel of the weighted Szeg\H{o} projection because for all $u\in L^2_{\varphi}(\partial\Omega)$ and $a\in\Omega$
\begin{equation}
(P_{\varphi}u)(a) = \langle P_{\varphi}u, S_{\varphi}(\cdot, a)\rangle_{\varphi}=\langle u, S_{\varphi}(\cdot, a)\rangle_{\varphi} = \int_{\zeta\in\partial\Omega}S_{\varphi}(a,\zeta)u(\zeta)\varphi(\zeta) ds.
\end{equation}
Note that for all $h\in H^2(\partial\Omega)$
\[
h(a) = (\mathcal{C}h)(a) = \langle h, C_a\rangle = \langle h, \varphi^{-1}C_a\rangle_{\varphi} =  \langle h, P_{\varphi}(\varphi^{-1}C_a)\rangle_{\varphi}.
\]
Therefore, the weighted Szeg\H{o} kernel $S_{\varphi}(\cdot,a)$ is also given by the weighted Szeg\H{o} projection of $\varphi^{-1}C_a$.
That is,
\begin{equation}
S_{\varphi}(z,a)=(P_{\varphi} (\varphi^{-1}C_a))(z)=\int_{\zeta\in\partial\Omega}S_{\varphi}(z,\zeta)C_a(\zeta) ds
\end{equation}
for every $z\in\Omega$. In \cite{JV}, the following \textit{weighted Kerzman--Stein formula}
\begin{equation}
P_{\varphi}(I+\mathcal{A}_{\varphi})=\mathcal{C}
\end{equation}
was shown to hold on $C^{\infty}(\partial\Omega)$, where $I$ is the identity and $\mathcal{A}_{\varphi} = \mathcal{C}-\mathcal{C}^{*}_{\varphi}$ is the \textit{weighted Kerzman--Stein operator}. The operator $\mathcal{A}_{\varphi}$ admits an integral representation with a smooth kernel $A_{\varphi}(\cdot,\cdot)\in C^{\infty}(\partial\Omega\times\partial\Omega)$, called the \textit{weighted Kerzman--Stein kernel}. More precisely,
\begin{equation}
(\mathcal{A}_{\varphi}u)(z)=\int_{\zeta\in\partial\Omega} A_{\varphi}(z,\zeta) u(\zeta)\varphi(\zeta)\,ds,
\end{equation}
where, for $z,\zeta\in\partial\Omega$ with $z\neq\zeta$,
\begin{equation}\label{kskernel}
A_{\varphi}(z,\zeta)=\frac{1}{2\pi i}\left(\frac{T(\zeta)\varphi^{-1}(\zeta)}{\zeta-z}-\frac{\overline{T(z)}\varphi^{-1}(z)}{\bar{\zeta}-\bar{z}}\right),
\end{equation}
which extends smoothly to $\partial\Omega\times\partial\Omega$.
Since the Cauchy transform maps $C^{\infty}(\partial\Omega)$ into itself, it follows from the weighted Kerzman--Stein formula that the weighted Szeg\H{o} projection $P_{\varphi}$ also maps $C^{\infty}(\partial\Omega)$ into itself. Thus, $S_{\varphi}(\cdot,a)\in A^{\infty}(\Omega)$. In fact, $S_{\varphi}(z, w) \in C^{\infty}\left(\left( \overline \Om \times \overline \Om \right) \setminus \Delta\right)$ where $\Delta = \{(z, z): z \in \pa \Omega\}$ is the diagonal in $\pa \Om \times \pa \Om$ as Theorem $3.2$ in \cite{JV} shows.

\medskip

A function $u$ in $L^2(\partial\Omega)$ is orthogonal to $H^2_{\varphi}(\partial\Omega)$ if and only if $u\varphi$ is orthogonal to $H^2(\partial\Omega)$. All the functions $v\in L^2(\partial\Omega)$ orthogonal to $H^2(\partial\Omega)$ are of the form $\overline{HT}$ where $H\in H^2(\partial\Omega)$. If $v\in C^{\infty}(\partial\Omega)$ then $H\in A^{\infty}(\Omega)$. So, the orthogonal decomposition of $\varphi^{-1}C_a$ is given by
\begin{equation}
\varphi^{-1}C_a = S_{\varphi}(\cdot,a) + \varphi^{-1} \overline{H_aT}
\end{equation}
where $H_a\in A^{\infty}(\Omega)$. Also, the above decomposition shows that $H_a$ is holomorphic in $a\in\Omega$ for fixed $z\in\partial\Omega$. The \textit{weighted Garabedian kernel} of $\Omega$ with respect to $\varphi$ is defined by
\begin{equation}
L_{\varphi}(z,a) = \frac{1}{2\pi}\frac{1}{z-a} - i H_a(z).
\end{equation}
For a fixed $a\in\Omega$, $L_{\varphi}(z,a)$ is a holomorphic function of $z$ on $\Omega\setminus\{a\}$ with a simple pole at $z=a$ with residue $1/2\pi$, and extends $C^{\infty}$ smoothly to $\partial\Omega$. In addition, $L_{\varphi}(z,a)$ is holomorphic in $a$ on $\Omega$ for fixed $z\in \partial\Omega$. Moreover, it is known that (see \cite{Nehari})
\begin{equation}
L_{\varphi}(z,a) = - L_{1/\varphi}(a,z)
\end{equation}
for $z,a\in\Omega$. Therefore, for a fixed $z\in\Omega$, $L_{\varphi}(z,a)$ is a holomorphic function of $a$ on $\Omega\setminus\{z\}$ with a simple pole at $a=z$ and residue $-{1}/2\pi$, and extends $C^{\infty}$ smoothly to $\partial\Omega$. Finally, for $z \in \partial \Omega$, $a \in \Omega$
\begin{eqnarray*}
S_{\varphi}(a,z) &=& \overline{S_{\varphi}(a,z)} 
=
\frac{1}{\varphi(z)}\left(\frac{1}{2\pi i}\frac{T(z)}{z-a}-H_a(z) T(z)\right)
\\&=&
\frac{1}{i\varphi(z)}\left(\frac{1}{2\pi}\frac{1}{z-a}-iH_a(z)\right)T(z)
\end{eqnarray*}
shows that the weighted Szeg\H{o} kernel and the weighted Garabedian kernel satisfy the identity
\begin{equation}
S_{\varphi}(a,z) = \frac{1}{i\varphi(z)} L_{\varphi}(z,a) T(z).
\end{equation}
It follows from Theorem $3.3$ in \cite{JV} that the function $l_{\varphi}(z,w)$ defined by 
\[
L_{\varphi}(z,w) = \frac{1}{2\pi(z-w)} + l_{\varphi}(z,w)
\]
is holomorphic on $\Omega\times\Omega$ and extends to be in $C^{\infty} (\overline{\Omega}\times\overline{\Omega})$.

	\section{Convergence of weighted Szeg\H{o} kernels}

    The weights for which the weighted Szeg\H{o} kernel is well-defined are studied in \cite{pw-um}. We will focus only on a special class of weights that are smooth and positive functions on the boundary and study the continuity of the map $\varphi \mapsto S_{\varphi}$. In \cite{JV}, it was shown that for a sequence of positive real-valued $C^{\infty}$ functions $\varphi_k$ on $\partial\Omega$ such that $\varphi_k\rightarrow\varphi$ in the $C^{\infty}$-topology on $\partial\Omega$,  
		\begin{equation}
		\lim_{k\rightarrow\infty}S_{\varphi_k}(z,w) = S_{\varphi}(z,w)
		\quad
		\text{and}
		\quad
		\lim_{k\rightarrow\infty}l_{\varphi_k}(z,w) = l_{\varphi}(z,w)
		\end{equation}
		locally uniformly on $(\Omega\times\overline{\Omega})\cup(\overline{\Omega}\times\Omega)$. We extend this to the maximal possible set that is expected, in view of the fact that each Szeg\H{o} kernel is in $C^{\infty}\left(\left( \overline \Om \times \overline \Om \right) \setminus \Delta\right)$.   

	\begin{thm}\label{Szego}
		Let $\Omega\subset\mathbb{C}$ be a $C^{\infty}$ smoothly bounded $n$-connected domain and $\varphi$ be a positive real-valued $C^{\infty}$-smooth function on $\partial\Omega$. 
		Let $\{\varphi_k\}_{k=1}^{\infty}$ be a sequence of positive real-valued $C^{\infty}$-smooth functions on $\partial\Omega$ such that $\varphi_k\rightarrow\varphi$ in the $C^{\infty}$ topology on $\partial\Omega$ as $k\rightarrow\infty$. Then
		\begin{equation}
		\lim_{k\rightarrow\infty}S_{\varphi_k}(z,w) = S_{\varphi}(z,w) 
		\end{equation}
		locally uniformly on $(\overline{\Om}\times\overline{\Om})\setminus \Delta$. Furthermore, all the derivatives of $S_{\varphi_k}(z,w)$ converge locally uniformly to the corresponding derivatives of $S_{\varphi}(z,w)$ on $(\overline{\Om}\times\overline{\Om})\setminus \Delta$.
	\end{thm}
	
	\begin{proof}
		Let $z_0,w_0\in\partial\Omega$ be distinct arbitrary points. Choose relatively compact neighborhoods $U, V\subset\mathbb{C}$ of $z_0, w_0$, respectively such that they both intersect only one of the boundary components of $\Omega$ namely, the ones containing $z_0, w_0$, and satisfy $\overline{U}\cap \overline{V} = \emptyset$. Let $\widetilde{U} = U\cap\overline{\Omega}$ and $\widetilde{V} = V\cap\overline{\Omega}$. It is sufficient to show that 
		\[
		\lim_{k\rightarrow\infty}S_{\varphi_k}(z,w) = S_{\varphi}(z,w) 
		\]
		uniformly on $\widetilde{U}\times\widetilde{V}$. Let $f:\Omega\rightarrow \mathbb{D}$ be a proper holomorphic map which is an $n$-to-one branched covering and which maps each boundary curve bijectively to the unit circle. Suppose that $f$ has zeros at $a_1,\ldots,a_N \in \Omega$ with multiplicities $M(1),\ldots, M(N)$ respectively. Note that $M(1)+\cdots+M(N) = n$. 
        
        \medskip
        
        For a smooth weight function $\psi$ and non-negative integers $\mu,\nu$, let
        \[
        S_{\psi}^{(\mu,\ov{\nu})}(z, w) 
        :=
        \frac{\partial^{\mu+\nu}}{\partial z^{\mu}\partial {\overline w}^{\nu}}S_{\psi}(z, w),
        \qquad
        S_{\psi}^{(\mu)}(z, w) := S_{\psi}^{\mu,\ov{0}}(z, w),
        \qquad
        S_{\psi}^{(\ov{\nu})}(z, w) := S_{\psi}^{0,\ov{\nu}}(z, w).
        \]
        By Bell \cite{Be1}, 
		\begin{equation}\label{rel_k}
			S_{\varphi_k}(z,w) = \frac{1}{1-f(z)\overline{f(w)}} \left(\sum_{i,j=1}^N \sum_{\mu=0}^{M(i)-1} \sum_{\nu=0}^{M(j)-1} c_{ij\mu\nu k}\, S_{\varphi_k}^{(\ov{\mu})}(z,a_i) \overline{S_{\varphi_k}^{(\ov{\nu})}(w,a_j)}\right) 
		\end{equation}
		where the constants $c_{ij\mu\nu k}$ are uniquely determined by the system of equations
		\begin{equation}\label{constants_k}
			\sum_{j=1}^N\sum_{\nu=0}^{M(j)-1}c_{ij\mu\nu k}\overline{S_{\varphi_k}^{(q,\ov{\nu})}(a_l,a_j)}
			=
			\left\{
			\begin{array}{lr}
				1 & \text{if } i=l \text{ and } \mu=q \\
				0 & \text{if } i\neq l \text{ or } \mu\neq q
			\end{array}
			\right.
		\end{equation}
		where $1\leq l\leq N$, $1\leq i\leq N$, $0\leq q\leq M(l)-1$ and $0\leq \mu\leq M(i)-1$. 
       Similarly, 
		\begin{equation}\label{rel}
			S_{\varphi}(z,w) = \frac{1}{1-f(z)\overline{f(w)}} \left(\sum_{i,j=1}^N \sum_{\mu=0}^{M(i)-1} \sum_{\nu=0}^{M(j)-1} c_{ij\mu\nu}\, S_{\varphi}^{(\ov{\mu})}(z,a_i) \overline{S_{\varphi}^{(\ov{\nu})}(w,a_j)}\right)
		\end{equation}
		where the constants $c_{ij\mu\nu}$ are uniquely determined by the system of equations
		\begin{equation}\label{constants}
			\sum_{j=1}^N\sum_{\nu=0}^{M(j)-1}c_{ij\mu\nu}\overline{S_{\varphi}^{(q,\ov{\nu})}(a_l,a_j)}
			=
			\left\{
			\begin{array}{lr}
				1 & \text{if } i=l \text{ and } \mu=q \\
				0 & \text{if } i\neq l \text{ or } \mu\neq q
			\end{array}
			\right.
		\end{equation}
		where $1\leq l\leq N$, $1\leq i\leq N$, $0\leq q\leq M(l)-1$ and $0\leq \mu\leq M(i)-1$.
		We can write the systems (\ref{constants_k}) and (\ref{constants}) of linear equations in matrix form as:
		\[
		\mathcal{A}_k\mathcal{C}_k = \mathcal{B}
		\quad {and} \quad
		\mathcal{A} \mathcal{C} = \mathcal{B}
		\]
		where $\mathcal{B}$ is an $n\times 1$ vector with entries $0$ or $1$, $\mathcal{A}_k$ is an $n\times n$ matrix with entries $S_{\varphi_k}^{(q,\ov{\nu})}(a_l,a_j)$, 
		$\mathcal{A}$ is same as the matrix $\mathcal{A}_k$ if we replace the entries $S_{\varphi_k}$ by $S_{\varphi}$, and 
		$\mathcal{C}_k$ and $\mathcal{C}$ are $n\times 1$ vectors with entries $c_{ij\mu\nu k}$ and $c_{ij\mu\nu}$ respectively. The fact that $M(1) + \cdots +M(N) = n$ is being used here.
		
		\medskip
		
		Since $S_{\varphi_k}\rightarrow S_{\varphi}$ locally uniformly (as do their derivatives) on $\Omega\times\Omega$ as $k\rightarrow\infty$, we have $\mathcal{A}_k\rightarrow\mathcal{A}$ as $k\rightarrow\infty$. Moreover, the matrices $\mathcal{A}_k$ and $\mathcal{A}$ are invertible. Since inversion is continuous on the space of invertible matrices, $\mathcal{A}_k^{-1}\rightarrow\mathcal{A}^{-1}$ as $k\rightarrow\infty$. Therefore, $c_{ij\mu\nu k}\rightarrow c_{ij\mu\nu}$ as $k\rightarrow\infty$.

        \medskip
		
		If $z_0$ and $w_0$ lie on the same boundary curve, then $1-f(z)\overline{f(w)}$ does not vanish on $\widetilde{U}\times\widetilde{V}$ because $f$ is bijective on each boundary component of $\partial \Omega$. 

        \medskip
        
        Let $\gamma_1,\ldots, \gamma_n$ denote $n$ boundary curves of $\Om$. When $z_0,w_0$ lie on different boundary curves, say $\gamma_1, \gamma_2$ respectively,  choose $z_1\in \gamma_1$ with $z_1\neq z_0$. Furthermore, put $z_2 = w_0$ and choose $z_i\in\gamma_i$ arbitrarily for $i=3,\ldots,n$. By Bell--Kaleem \cite{BeKf}, there exists a proper holomorphic map $f: \Omega \rightarrow \mathbb D$ which maps $z_i$ to $1 \in \partial \mathbb D$ for all $1\leq i\leq n$, and is bijective when restricted to each boundary component of $\partial \Omega$ onto $\partial\mathbb{D}$. However, we only need a proper holomorphic map $f$ that separates $z_0$ and $w_0$. Indeed, $f(z_0)\neq 1$ and $f(w_0) = 1$. By continuity of $f$ on $\ov \Om$, we may shrink $U$ and $V$, if necessary, so that $1-f(z)\overline{f(w)}$ does not vanish on $\widetilde{U}\times\widetilde{V}$. 
		
		\medskip
		
		Since the zeros $a_i$ lie in $\Omega$, \cite{JV} implies that
		\[
		\lim_{k\rightarrow\infty}S_{\varphi_k}^{(\ov{\mu})}(z,a_i)= S_{\varphi}^{(\ov{\mu})}(z,a_i)
		\quad
		\text{and}
		\quad
		\lim_{k\rightarrow\infty}S_{\varphi_k}^{(\ov{\nu})}(w,a_j)= S_{\varphi}^{(\ov{\nu})}(w,a_j)
		\]
		uniformly for $z\in \widetilde{U}$ and $w\in\widetilde{V}$ respectively. These observations along with (\ref{rel_k}) and (\ref{rel}) show that 
		\[
		\lim_{k\rightarrow\infty}S_{\varphi_k}(z,w) = S_{\varphi}(z,w)
		\]
		uniformly on $\widetilde{U}\times \widetilde{V}$. The local uniform convergence of the derivatives follows by applying the same arguments to $S_{\varphi_k}^{(\mu,\ov \nu)}(z,w)$ and $S_{\varphi}^{(\mu,\ov \nu)}(z,w)$ after differentiating Bell's formulas \eqref{rel_k} and \eqref{rel}, where $\mu,\nu$ are non-negative integers.
	\end{proof} 

\begin{thm}\label{gara}
		Let $\Omega\subset\mathbb{C}$ be a $C^{\infty}$ smoothly bounded $n$-connected domain and $\varphi$ be a positive real-valued $C^{\infty}$-smooth function on $\partial\Omega$. 
		Let $\{\varphi_k\}_{k=1}^{\infty}$ be a sequence of positive real-valued $C^{\infty}$-smooth functions on $\partial\Omega$ such that $\varphi_k\rightarrow\varphi$ in $C^{\infty}$ topology on $\partial\Omega$ as $k\rightarrow\infty$. Then
		\begin{equation}
		\lim_{k\rightarrow\infty}l_{\varphi_k}(z,w) = l_{\varphi}(z,w)
		\end{equation}
		uniformly on $\overline{\Om}\times\overline{\Om}$. Furthermore, all the derivatives of $l_{\varphi_k}(z,w)$ converge uniformly to the corresponding derivatives of $l_{\varphi}(z,w)$ on $\overline{\Om}\times\overline{\Om}$.
	\end{thm}
	
\begin{proof}
For $w\in\Om$ and $\zeta\in\partial\Om$, define
\[
G_w(\zeta) = \frac{1}{2\pi(\zeta-w)}.
\]
Put $\psi_k : = 1/\varphi_k$ and $\psi: = 1/\varphi$. From \cite{JV}, we know that
\begin{equation}
l_{\varphi_k}(z,w) = (P_{\psi_k}\mathcal{A}_{\psi_k} G_w) (z)
\qquad
\text{and}
\qquad
l_{\varphi}(z,w) = (P_{\psi}\mathcal{A}_{\psi} G_w) (z).
\end{equation}
For $w\in\Omega$ and $\zeta\in\partial\Omega$,
\begin{eqnarray*}
(\mathcal{A}_{\psi}G_{w})(\zeta)&=&\int_{\xi\in\partial\Omega} A_{\psi}(\zeta,\xi) \, G_{w}(\xi) \,ds
=
\frac{1}{2 \pi}\int_{\xi\in\partial\Omega} A_{\psi}(\zeta,\xi)\left(\frac{1}{\xi-w}\right) ds\\
&=&
\frac{1}{2\pi i}\int_{\xi\in\partial\Omega} \frac{i\,A_{\psi}(\zeta,\xi)\,\ov{T(\xi)}}{\xi-w} d\xi
=
(\mathcal{C} f_{\zeta})(w),
\end{eqnarray*}
where $f_{\zeta}(\xi) = i\,A_{\psi}(\zeta,\xi)\,\ov{T(\xi)}$ for $\xi\in\partial\Om$. Since $f_{\zeta}\in C^{\infty}(\partial\Omega)$, the function $\mathcal{C}f_{\zeta}\in A^{\infty}(\Omega)$. Therefore, we have $(\mathcal{A}_{\psi}G_{w})(\zeta)\in C^{\infty}(\overline{\Omega})$ as a function of $w$ for every $\zeta\in\partial\Omega$. Similarly, we can write
$
(\mathcal{A}_{\psi_k}G_{w})(\zeta)
=
(\mathcal{C} f^k_{\zeta})(w),
$
where $f^k_{\zeta}(\xi) = i\,A_{\psi_k}(\zeta,\xi)\,\ov{T(\xi)}$ for $\xi\in\partial\Om$. The same regularity holds for $(\mathcal{A}_{\psi_k}G_{w})(\zeta)$.

\medskip

For a non-negative integer $s$, define the norm $\Vert u\Vert_s$ of a function $u$ defined on the boundary of $\Omega$ as
\[
\Vert u\Vert_s = \sup\left\lbrace\left\lvert \frac{d^m}{dt^m}u(z(t))\right\rvert : t\in J, \, \, 0\leq m\leq s\right\rbrace.
\]
For every non-negative integer $s$, there exists a positive integer $n=n(s)$ and a constant $K=K(s)>0$ such that $\Vert \mathcal{C}u\Vert_s\leq K \Vert u\Vert_{n}$ for all $u\in C^{\infty}(\partial\Omega)$, see \cite[Theorem 9.2]{Be0}. Put $s=0$.
For $w\in\overline{\Omega}$ and $\zeta\in\partial\Omega$,
\begin{eqnarray*}
\vert(\mathcal{A}_{\psi_k}G_{w})(\zeta)-(\mathcal{A}_{\psi}G_{w})(\zeta)\vert 
&=&
\vert\mathcal{C}(f^k_{\zeta}-f_{\zeta})\vert(w)
\\
&\leq&
\Vert\mathcal{C}(f^k_{\zeta}-f_{\zeta})\Vert_0
\leq
K\Vert f^k_{\zeta}-f_{\zeta}\Vert_n.
\end{eqnarray*}
It can be easily checked that $\Vert f^k_{\zeta}-f_{\zeta}\Vert_n\rightarrow 0$ as $k\rightarrow\infty$. Thus, $(\mathcal{A}_{\psi_k}G_{w})(\zeta)\rightarrow (\mathcal{A}_{\psi}G_{w})(\zeta)$ uniformly on $\partial\Omega\times \overline{\Omega}$ as $k\rightarrow\infty$ when considered as a function of $(\zeta,w)$. Similarly, for a non-negative integer $m$, it can be shown that
\[
\frac{\partial^m}{\partial t^m} (\mathcal{A}_{\psi_k}G_{w})(\zeta(t))
\rightarrow
\frac{\partial^m}{\partial t^m}(\mathcal{A}_{\psi}G_{w})(\zeta(t))
\]
uniformly on $\partial\Omega\times \overline{\Omega}$ as $k\rightarrow\infty$ when considered as a function of $(\zeta,w)$. Thus, for every non-negative integer $s$, we have proved
\begin{equation}\label{aid_l}
\Vert \mathcal{A}_{\psi_k} G_{w}-\mathcal{A}_{\psi} G_{w}\Vert_s\rightarrow 0
\qquad
\text{as }k\rightarrow\infty
\end{equation}
uniformly for all $w\in\overline{\Omega}$. 

\medskip

Given a non-negative integer $s$, there exists a positive integer $n=n(s)$ and a constant $C(s,k)>0$ such that $\Vert P_{\psi_k}u\Vert_s \leq C(s,k) \Vert u\Vert_n$ for all $u\in C^{\infty}(\partial\Omega)$, see \cite[Theorem 3.1]{JV}. It can be seen from its proof that
\[
C(s,k)
=
K+
\sup\left\lbrace\left\lvert \frac{d^m}{dt^m} A_{\psi_k}(z(t),\zeta) \right\rvert:t\in J,\,0\leq m\leq s,\,\zeta\in\partial\Omega\right\rbrace 
\int_{\zeta\in\partial\Omega}\psi_k(\zeta)\,ds
\]
for some constant $K=K(s)>0$. Since $\psi_k\rightarrow\psi$ uniformly on $\partial\Omega$ and $A_{\psi_k}\rightarrow A_{\psi}$ uniformly on $\partial\Omega\times\partial\Omega$, the constants $C(s,k)$ are bounded.
Therefore, there exists a constant $M = M(s)>0$ such that
\begin{equation}\label{proj0}
\Vert P_{\psi_k}u\Vert_s \leq M \Vert u\Vert_n
\end{equation}
for all  $u\in C^{\infty}(\partial\Omega)$ and $k\geq 1$. Now for $u\in C^{\infty}(\Om)$ and every $h\in H^2(\partial\Om)$, we have 
\begin{eqnarray*}
\langle
P_{\psi_k}u - P_{\psi}u, h 
\rangle_{\psi_k}
&=&
\langle
P_{\psi_k}(u - P_{\psi}u), h 
\rangle_{\psi_k}
=
\langle
u - P_{\psi}u, h 
\rangle_{\psi_k}
\\&=&
\langle
(I - P_{\psi})u, h 
\rangle_{\psi_k} 
-
\langle
(I - P_{\psi})u, h 
\rangle_{\psi}
\\&=&
\left\langle
\frac{\psi_k- \psi}{\psi_k}(I - P_{\psi})u, h 
\right\rangle_{\psi_k}.
\end{eqnarray*}
Therefore,
\begin{equation}\label{proj}
P_{\psi_k}u - P_{\psi}u
=
P_{\psi_k}\left(
\frac{\psi_k- \psi}{\psi_k}(I - P_{\psi})u
\right).
\end{equation}
Finally, for $w\in\ov{\Om}$
\begin{eqnarray*}
l_{\varphi_k}(\cdot,w) - l_{\varphi}(\cdot, w)
&=&
P_{\psi_k}\mathcal{A}_{\psi_k} G_w
-
P_{\psi}\mathcal{A}_{\psi} G_w
\\&=&
P_{\psi_k} \left(
\mathcal{A}_{\psi_k} G_w - \mathcal{A}_{\psi} G_w
\right)
+ 
\left(
P_{\psi_k} - P_{\psi}
\right) \mathcal{A}_{\psi} G_w
\\&=&
P_{\psi_k} \left(
\mathcal{A}_{\psi_k} G_w - \mathcal{A}_{\psi} G_w
\right)
+ 
P_{\psi_k}\left(
\frac{\psi_k- \psi}{\psi_k}(I - P_{\psi})\mathcal{A}_{\psi} G_w
\right).
\end{eqnarray*}
Therefore, for a non-negative integer $s$, we have
\begin{eqnarray*}
\Vert l_{\varphi_k}(\cdot,w) - l_{\varphi}(\cdot, w)\Vert_s
&\leq &
\left\lVert P_{\psi_k} \left(
\mathcal{A}_{\psi_k} G_w - \mathcal{A}_{\psi} G_w
\right)\right\rVert_s
+
\left\lVert
P_{\psi_k}\left(
\frac{\psi_k- \psi}{\psi_k}(I - P_{\psi})\mathcal{A}_{\psi} G_w
\right)
\right\rVert_s
\\&\leq &
M
\left(
\left\lVert 
\mathcal{A}_{\psi_k} G_w - \mathcal{A}_{\psi} G_w
\right\rVert_n
+
\left\lVert
\frac{\psi_k- \psi}{\psi_k}(I - P_{\psi})\mathcal{A}_{\psi} G_w
\right\rVert_n
\right).
\end{eqnarray*}
Since $\psi_k\rightarrow\psi$ in $C^{\infty}$ topology on $\partial\Om$ as $k\rightarrow\infty$, we use \eqref{aid_l} to obtain that
\[
\Vert l_{\varphi_k}(\cdot,w) - l_{\varphi}(\cdot, w)\Vert_s \rightarrow 0 
\qquad
\text{as }k\rightarrow\infty
\]
uniformly for all $w\in\ov{\Om}$. This proves the theorem.
\end{proof}

\begin{rem}
The arguments leading to \eqref{proj0} and \eqref{proj}
also establish a convergence result for the weighted
Szeg\H{o} projections. Indeed, under the hypotheses of
Theorem \ref{gara}, for every
$u\in C^\infty(\partial\Omega)$, we have
\[
\lim_{k\rightarrow\infty}P_{\varphi_k}u=P_\varphi u
\]
in the $C^\infty$ topology on $\overline{\Omega}$.
\end{rem}

		
\section{zeros of the weighted Szeg\H{o} kernel}
	
\noindent In the unweighted case, the following behaviour of the zeros of the Szeg\H{o} kernel is known and an exposition can be found in Bell \cite{Be0}.
	
	\begin{thm}[Bell]\label{Bell_zeros}
		Let $\Omega\subset\mathbb{C}$ be a bounded $n$-connected domain with $C^{\infty}$-smooth boundary curves $\gamma_i$, $1 \le i \le n$. Let $w_j$ be a sequence in $\Omega$ that tends to a point $a \in \gamma_m$. As $w_j$ tends to $a$, the zeros $Z^i(w_j)$ of $S(z,w_j)$ become simple zeros, and it is possible to order them so that for each $i$, $i\neq m$, there is a point $Z^i(a)\in\gamma_i$ such that $Z^i(w_j)$ tends to $Z^i(a)$. 
		
		\medskip
		
		$S(z,a)$ is non-vanishing on $\Om$ as a function of $z$ and has exactly $n-1$ zeros $Z^i(a)$ on $\partial\Om$, one on each boundary component not containing $a$. Furthermore, the zeros are simple in the sense that $S'(Z^i(a),a)\neq 0$ for each zero $Z^i(a)$. 
	\end{thm}
	
	\begin{rem}
		It is known that $L(z, w)$ does not vanish for $z\in \overline{\Om}, w\in\Om$. From the above theorem, given $a\in\partial\Om$, $L(z,a)$ is non-vanishing for $z\in\Om$ and has exactly $n-1$ zeros on $\partial\Om$, which are same as the zeros of $S(z,a)$.
	\end{rem}
	
	We shall now study the behaviour of zeros of the weighted Szeg\H{o} kernel and the weighted Garabedian kernel as weights vary.
	
    \begin{defn}
    A domain $\Om\subset\mathbb{C}$ is said to have real analytic boundary if its boundary can be (locally) parametrized by a function $z(t) = x(t) + i y(t)$, where the real-valued functions $x(t)$ and $y(t)$ are real analytic functions of $t$. A function $u$ defined on $\partial\Om$ is said to be real analytic if $u(z(t))$ is a real analytic function of $t$.
    \end{defn}
	
	Let $\Omega\subset\mathbb{C}$ be a bounded $n$-connected domain whose boundary curves are {\it real analytic} with no isolated boundary point. If $\varphi$ is a positive real analytic function on $\partial\Omega$, then the weighted Szeg\H{o} kernel $S_{\varphi}(z,w)$ extends to a neighborhood of $(\overline{\Om}\times \overline{\Om})\setminus\Delta$ as a function that is holomorphic in $z$ and antiholomorphic in $w$, and the weighted Garabedian kernel is given by
	\[
	L_{\varphi}(z,w) = \frac{1}{2\pi(z-w)}+ l_{\varphi}(z,w)
	\]
	where $l_{\varphi}(z,w)$ extends to be holomorphic in $z$ and $w$ on a neighborhood of ($\overline{\Om}\times\ov{\Om})$. 
    
    \medskip
    
    The real analyticity of the boundary of $\Om$ and of the weight function $\varphi$ is used to establish the meromorphic extension of the weighted Szeg\H{o} and weighted Garabedian kernels across the boundary of $\Om$; see \cite[Theorem 26.5]{Be0} and \cite[Section 7]{Be1}. Fix $w\in\Om$. Let $z_1,\ldots,z_N$ be the zeros of $S_{\varphi}(\cdot,w)$ in $\Om$ and $b_1,\ldots,b_M$ be the zeros of $S_{\varphi}(\cdot, w)$ on $\partial\Om$. We know that
	\begin{equation}\label{Bdy}
		\overline{S_{\varphi}(z,w)} = \frac{1}{i\varphi(z)}L_{\varphi}(z,w)T(z),
		\quad\,\,
		w\in\Om, z\in\partial\Omega.
	\end{equation}
	Therefore, the zeros of $L_{\varphi}(\cdot, w)$ on $\partial\Om$ are $b_1,\ldots, b_M$. Let $p_1,\ldots, p_Q$ be the zeros of $L_{\varphi}(\cdot,w)$ in $\Om$. Let $h = S_{\varphi}(\cdot,w)L_{\varphi}(\cdot,w)$, which is a meromorphic function in a neighborhood of $\overline{\Om}$. 
    
    \medskip
    
    The generalized argument principle in \cite[Chapter 13]{Be0} says that zeros or poles in the interior of $\Om$ contribute to the increment of the argument of $h$ as usual, whereas zeros and poles on the boundary contribute half the normal amount to the increment. Let $m_h(z)$ denote the multiplicity of a zero or pole of $h$ at $z$. Then
	\[
	\sum_{i=1}^N m_h(z_i) + \sum_{i=1}^Q m_h(p_i) - m_h(w) + \frac{1}{2}\sum_{i=1}^M m_h(b_i)
	=
	\frac{1}{2\pi} \Delta \arg h.
	\]
	Let $m_S(z)$ and $m_L(z)$ denote the multiplicity of a zero or a pole of $S_{\varphi}(\cdot,w)$ and $L_{\varphi}(\cdot, w)$ respectively at $z$. Then,
	\[
	\sum_{i=1}^N m_S(z_i) + \sum_{i=1}^Q m_L(p_i) - m_L(w) + \sum_{i=1}^M m_S(b_i) 
	=
	\frac{1}{2\pi} \Delta \arg S_{\varphi}(\cdot,w)L_{\varphi}(\cdot,w).
	\]
	Multiplying $S_{\varphi}(z,w)$ on both sides of (\ref{Bdy}), we obtain
	\[
	\frac{1}{i\varphi(z)}S_{\varphi}(z, w) L_{\varphi}(z,w) T(z)= \vert S_{\varphi}(z,w)\vert^2.
	\]
	Therefore, 
	\[
	\Delta \arg S_{\varphi}(\cdot,w)L_{\varphi}(\cdot,w) = - \Delta \arg T = 2\pi(n-2).
	\]
	Since $m_L(w) = 1$, we have
	\begin{equation}
	\sum_{i=1}^N m_S(z_i) + \sum_{i=1}^Q m_L(p_i) + \sum_{i=1}^M m_S(b_i) 
	=
	n-1.
	\end{equation}
    
	\begin{obs}
		Thus, for $w\in\Omega$, the combined number of zeros of $S_{\varphi}(\cdot, w)$ and $L_{\varphi}(\cdot, w)$ in $\Om$, plus the number of zeros of $S_{\varphi}(\cdot,w)$ on $\partial\Om$ (which are also the zeros of $L_{\varphi}(\cdot,w)$ on $\partial\Om$), equals $n-1$. Furthermore, the zeros are counted with multiplicity.
	\end{obs}
	
\begin{defn}[Double of the domain $\Om$]
Consider the domain $\Om$ with the coordinate chart $z$ mapping $\Om$ into the complex plane. Let $\widetilde{\Om}$ be a copy of $\Om$ with the coordinate chart $\ov{z}$. The double $\widehat{\Om}$ of $\Om$ is obtained by gluing $\Om$ to $\widetilde{\Om}$ along corresponding boundary curves via coordinate charts described as follows. 

\medskip

For each $p\in\partial\Om$, choose a sufficiently small neighborhood $U_p$ of $p$ in $\mathbb{C}$ such that there exists a biholomorphism $F_p:U_p\rightarrow \mathbb{D}$, with $F_p(U_p\cap \Om) = \mathbb{D}^+ = \{w\in\mathbb{D} : Im w>0\}$ and $F_p(U_p\cap \partial\Om) = (-1,1)$. Such a map exists because $\partial\Om$ is real-analytic. Define the coordinate chart $\Phi_p$ on $\widehat{\Om}$ around $p$ by
\[
\Phi_p(z) =
\begin{cases}
F_p(z), & z\in \overline{\Om} \cap U_p
\\
\overline{F_p(z)}, & otherwise
\end{cases}.
\]
The transition functions are holomorphic using Schwarz reflection theorem.
Thus, $\widehat{\Omega}$ is a compact Riemann surface containing $\Omega$ as an open subset and admitting a canonical antiholomorphic involution $R:\widehat{\Omega}\to\widehat{\Omega}$ that interchanges the copies $\Om$ and $\widetilde{\Om}$, and whose fixed-point set is precisely $\partial\Omega$.
\end{defn}

It is known (see \cite{Be0}) that if $g$ and $h$ are meromorphic functions on $\Omega$ which extend continuously to the boundary such that $g(z) = \overline{h(z)}$ for $z\in \partial\Omega$, then $g$ extends meromorphically to the double $\widehat{\Omega}$ with the extension $\hat{g}$ given by
\[
\hat{g}(z) = 
\begin{cases}
g(z) & z\in \Omega \sqcup \partial\Omega\\
h(R(z)) & z\in \widetilde{\Omega}
\end{cases}.
\]

\medskip
	
Let $\{\varphi_k\}_{k=1}^{\infty}$ be a sequence of positive real analytic functions on $\partial\Omega$ such that $\varphi_k\rightarrow 1$ in $C^{\infty}$ topology on $\partial\Omega$ as $k\rightarrow\infty$. Let $a\in\partial\Omega$. 
	
	\begin{thm}	
		There exists $k_0 \ge 1$ such that for all $k\geq k_0$ the following holds: 
		
		\medskip
		
		Let $w_j\in\Om$ be a sequence that converges to $a$. The functions $S_{\varphi_k}(\cdot,w_j)$ and $L_{\varphi_k}(\cdot,w_j)$ have distinct zeros in $\Omega$ for large enough $j$. The combined zeros $Z^i_k(w_j)$ of $S_{\varphi_k}(z,w_j)$ and $L_{\varphi_k}(\cdot, w_j)$ in $\overline{\Omega}$, become simple zeros, and it is possible to order them so that for each $i=1,\ldots, n-1$, there is a point $Z_k^i(a)\in\ov{\Om}$ such that $Z^i_k(w_j)$ tends to $Z^i_k(a)$ as $j\rightarrow\infty$. 
		
		\medskip
		
		The functions $S_{\varphi_k}(\cdot,a)$ and $L_{\varphi_k}(\cdot,a)$ have distinct zeros in $\Omega$. The combined number of zeros of $S_{\varphi_k}(\cdot, a)$ and $L_{\varphi_k}(\cdot,a)$ in $\Om$, plus the number of zeros of $S_{\varphi_k}(\cdot,a)$ on $\partial\Om$ $($which are also the zeros of $L_{\varphi_k}(\cdot,a)$ on $\partial\Om)$, counting multiplicity, equals $n-1$, and the zeros are the same as $Z_{k}^i(a)$. Furthermore, all $Z_{k}^i(a)$ are distinct for $i=1,\ldots, n-1$, and therefore the zeros are simple.
		
		\medskip
		
		Moreover, 
		\[
		\lim_{j,k\rightarrow\infty} Z^i_{k}(w_j) = Z^i(a)
		\]
		where $Z^i(a)$ are the zeros of $S(\cdot,a)$ (and $L(\cdot,a)$ as well).
	\end{thm}
	
	\begin{proof}
		Let $\gamma_i$ denote the $n$ boundary curves of $\Om$. For the sake of simplicity, we denote the boundary curves in such a manner that $a\in \gamma_n$. 
		For each $i\neq n$, choose a smoothly bounded simply connected neighborhood $U_i$ of $Z^i(a)$ in $\widehat{\Om}$ and a point $b\in \Om$ such that 
		
		\begin{enumerate}
			\item[(i)] $\overline{U_i}$ are mutually disjoint
			\item[(ii)] $z\in U_i$ if and only if $R(z)\in U_i$
			\item[(iii)] $\ov{U_i}\cap ({\gamma_n \cup \{b\}}) = \emptyset$ 
			\item[(iv)] $S(z,b) \neq 0$ for all $z\in \overline{U_i}\cap \ov{\Om}$.
		\end{enumerate}
		
		\medskip
		
		Let $C_i$ denote the boundary curve of $U_i$. We claim that there exists a neighborhood $U$ of $a$ in $\overline{\Om}$ such that $S(z,w)\neq 0$ for all $z\in C_i\cap \overline{\Om}$ and $w\in \ov{U}$. If it were not true, then for every positive integer $m$, there would exist $w_m\in\ov{\Om}$ with $\vert w_m - a\vert \leq 1/m$ and $z_m\in C_i\cap \overline{\Om}$ for some $i$ such that $S(z_m,w_m) = 0$. By the pigeonhole principle, infinitely many $z_m$ lie on some $C_{i_0}$. So, there exists a subsequence $z_{m_r}$ with $z_{m_r}\rightarrow z_0\in C_{i_0}$ as $r\rightarrow\infty$. This implies that $S(z_0,a) = 0$ which is a contradiction.
		
		\medskip
		
		We know that
		\begin{equation}\label{thms}
			\lim_{k\rightarrow\infty} S_{\varphi_k}(z,w) = S(z,w)
			\quad \text{and}\quad
			\lim_{k\rightarrow\infty} L_{\varphi_k}(z,w) = L(z,w)
		\end{equation}
		locally uniformly on $(\ov{\Om}\times\ov{\Om})\setminus \{(z,z): z\in\partial\Om\}$ and $(\ov{\Om}\times\ov{\Om})\setminus \{(z,z): z\in \overline{\Om}\}$ respectively. Therefore, we can choose an integer $k_0 \ge 1$ such that for all $k\geq k_0$, 
		
		\begin{enumerate}
			\item[(i)] $S_{\varphi_k}(z,w)$ does not vanish for $z\in C_i\cap \overline{\Om}$, $w\in \ov{U}$
			\item[(ii)] $L_{\varphi_k}(R(z),w)$ does not vanish for $z\in C_i\cap (\widetilde{\Om}\cup \partial\Om)$, $w\in \ov{U}$
			\item[(iii)] $S_{\varphi_k}(z,b)$ does not vanish for $z\in \ov{U_i}\cap \overline{\Om}$
			\item[(iv)] $L_{\varphi_k}(R(z),b)$ does not vanish for $z\in \ov{U_i}\cap (\widetilde{\Om}\cup \partial\Om)$.
		\end{enumerate}
		
		\medskip
		
		For $k\geq k_0$ and $w\in \ov{U}$, define the functions $G_k(z,w)$ by
		\begin{equation}
		G_k(z,w)
		=
		\begin{cases}
			\mathcal{S}_{\varphi_k}(z,w)& z\in U_i\cap \overline{\Omega}
			\\
			\overline{\mathcal{L}_{\varphi_k}(R(z),w)}& z\in U_i\cap \widetilde{\Omega}
		\end{cases}
		\end{equation}
		where 
		\begin{equation}
		\mathcal{S}_{\varphi_k}(z,w) = \frac{S_{\varphi_k}(z,w)}{S_{\varphi_k}(z,b)}
		\quad
		\text{and}
		\quad
		\mathcal{L}_{\varphi_k}(z,w) = \frac{L_{\varphi_k}(z,w)}{L_{\varphi_k}(z,b)}.
		\end{equation}
		Recall that
		\[
			\overline{S_{\varphi_k}(z,w)} = \frac{1}{i\varphi_k(z)}L_{\varphi_k}(z,w)T(z),
			\quad\,\,
			w\in\Om, z\in\partial\Omega.
		\]
		This implies that
		\[
		\mathcal{S}_{\varphi_k}(z,w)
		=
		\frac{S_{\varphi_k}(z,w)}{S_{\varphi_k}(z,b)}
		=
		\overline{\left(\frac{L_{\varphi_k}(z,w)}{L_{\varphi_k}(z,b)}\right)}
		=
		\overline{\left(\frac{L_{\varphi_k}(R(z),w)}{L_{\varphi_k}(R(z),b)}\right)}
		=
		\overline{\mathcal{L}_{\varphi_k}(R(z),w)},
		\quad
		z\in\partial\Om.
		\]
		Therefore, the functions $G_k$ are well-defined, holomorphic on $U_i$ in the first variable and extend smoothly to $C_i$. Similarly, we define the function $G(z,w)$ by
		\begin{equation}
		G(z,w)
		=
		\begin{cases}
			\mathcal{S}(z,w)& z\in U_i\cap \overline{\Omega}
			\\
			\overline{\mathcal{L}(R(z),w)}& z\in U_i\cap \widetilde{\Omega}
		\end{cases}
		\end{equation}
		where 
		\begin{equation}
		\mathcal{S}(z,w) = \frac{S(z,w)}{S(z,b)}
		\quad
		\text{and}
		\quad
		\mathcal{L}(z,w) = \frac{L(z,w)}{L(z,b)}.
		\end{equation}
		The function $G$ is well-defined, holomorphic on $U_i$ in the first variable, and extends smoothly to $C_i$. 
		Now, the function 
		\begin{equation}
		N_{ki}(w) = \frac{1}{2\pi i}\int_{C_i} \frac{{G_k}'(z, w)}{G_k(z,w)} dz
		\end{equation}
		converges uniformly to 
		\begin{equation}
		N_i(w)=\frac{1}{2\pi i} \int_{C_i}  \frac{{G}'(z, w)}{G(z,w)} dz
		\end{equation}
		for $w\in \overline{U}$ as $k\rightarrow\infty$. Here, $N_{ki}(w)$ and $N_i(w)$ are the number of zeros of $G_k(\cdot, w)$ and $G(\cdot,w)$ in $U_i$, respectively. We go back and choose a larger $k_0 \ge 1$, if needed, so that
		\[
		\vert N_{ki}(w) - N_i(w)\vert < 1/2
        \]
		for every $i,w\in \overline{U}$ and $k\geq k_0$.

        \medskip
		
		We shall now fix a $k\geq k_0$. Let $w_j\in\Omega$ be a sequence that converges to $a$. Assume, without loss of generality, that $w_j\in U$ for all $j$. The zeros of $G(\cdot, w_j)$ in $U_i$ are same as the zeros of $S(\cdot, w_j)$. By Theorem \ref{Bell_zeros}, there exists $j_0 \ge 1$ such that $N_i(w_j) = 1$ for all $i$ and $j\geq j_0$. Therefore, $N_{ki}(w_j) = 1$ for all $i$ and $j\geq j_0$. By the continuity of $N_{ki}(w)$ as a function of $w$, this implies that $N_{ki}(a) = 1$.
		
		\medskip
		
		The zeros of $G_k(\cdot,w_j)$ are the zeros of $S_{\varphi_k}(\cdot,w_j)$ or $L_{\varphi_k}(R(\cdot),w_j)$. The combined number of zeros of $S_{\varphi_k}(\cdot, w_j)$ and $L_{\varphi_k}(\cdot, w_j)$ in $\Om$, plus the number of zeros of $S_{\varphi_k}(\cdot,w_j)$ on $\partial\Om$ (which are also the zeros of $L_{\varphi_k}(\cdot,w_j)$ on $\partial\Om$), equals $n-1$. Since $N_{ki}(w_j) =1$ for all $i=1,\ldots,n-1$ and $j\geq j_0$, the functions $S_{\varphi_k}(\cdot,w_j)$ and $L_{\varphi_k}(\cdot,w_j)$ must have distinct zeros in $\Om$. Furthermore, the combined zeros $Z^i_k(w_j)$ of $S_{\varphi_k}(\cdot,w_j)$ and $L_{\varphi_k}(\cdot,w_j)$ in $\overline{\Om}$, are simple for all $j\geq j_0$. We order these zeros in such a manner that $Z^i_k(w_j)\in U_i$ for all $j\geq j_0$.
		
		\medskip
		
		Let $Z_i(G_k(\cdot,w_j))$ and $Z_i(G_k(\cdot,a))$ denote the zero of $G_k(\cdot,w_j)$ and $G_k(\cdot,a)$ in $U_i$ respectively. By an application of the residue theorem and (\ref{thms}),
		\[
		Z_i(G_k(\cdot,w_j))= \frac{1}{2\pi i} \int_{C_i} \frac{z\,{G_k}'(z, w_j)}{G_k(z,w_j)} dz 
		\quad 
		\rightarrow
		\quad
		\frac{1}{2\pi i} \int_{C_i}  \frac{z\,{G_k}'(z, a)}{G_k(z,a)} dz 
		=
		Z_i(G_k(\cdot,a))
        \]
		as $j\rightarrow\infty$.

    \medskip
        
		If $Z_i(G_k(\cdot,a))\in \ov{\Om}$, we let $Z^i_k(a) = Z_i(G_k(\cdot,a))$. Otherwise, let $Z^i_k(a) = R(Z_i(G_k(\cdot,a)))$. Then, 
		\[
		\lvert Z^{i}_k(w_j)- Z^i_k(a) \vert 
		\leq 
		\vert Z_i(G_k(\cdot,w_j))- Z_i(G_k(\cdot,a)) \vert
		\rightarrow 0
		\quad
		\text{as }j\rightarrow\infty.
		\]
		Furthermore, $Z^i_k(a)\in U_i$ are distinct.
		
		\medskip
		
		Since the zeros of $G_k(\cdot,a)$ are the zeros of $S_{\varphi_k}(\cdot,a)$ or $L_{\varphi_k}(R(\cdot),a)$, the points $Z^i_k(a)$ are the zeros of $S_{\varphi_k}(\cdot,a)$ or $L_{\varphi_k}(\cdot,a)$ in $\overline{\Om}$. Let, if possible, $z_0\in\ov{\Om}$ be another zero of $S_{\varphi_k}(\cdot,a)$ or $L_{\varphi_k}(\cdot,a)$. Let us assume, without loss of generality, that $S_{\varphi_k}(z_0,a) = 0$, as a similar reasoning can be applied if $L_{\varphi_k}(z_0,a)=0$.  It is clear that $z_0\notin \overline{(U_i\cap \Om)}$ for any $i=1,\ldots,n-1$. If $z_0\in\Om$, choose a ball $B(z_0,\delta)\subset\subset\Om$ centered at $z_0$ of radius $\delta>0$ such that 
		
		\begin{itemize}
			\item[(i)] $\overline{B(z_0,\delta)}$ is disjoint from $\ov{U_i\cap \Om}$
			\item[(ii)] $S_{\varphi_k}(\cdot,a)$ is non-vanishing on $C=\partial B(z_0,\delta)$. 
		\end{itemize}
		
		\noindent Then $S_{\varphi_k}(\cdot,w_j)$ is also non-vanishing on $C$ for all large $j$. If it were not true, then for every $j \ge 1$, there would exist $z_j\in C$ such that $S_{\varphi_k}(z_j,w_j) = 0$. By passing to a convergent subsequence of $z_j$, we would obtain a zero of $S_{\varphi_k}(\cdot,a)$ on $C$, which is not possible. As $j\rightarrow\infty$,
		\[
		\frac{1}{2\pi i}\int_{C}\frac{{S_{\varphi_k}}'(z,w_j)}{S_{\varphi_k}(z,w_j)} dz
		\quad
		\rightarrow
		\quad
		\frac{1}{2\pi i}\int_{C}\frac{{S_{\varphi_k}}'(z,a)}{S_{\varphi_k}(z,a)} dz
		\quad
		\geq 1.
		\]
		So, there must exist a zero of $S_{\varphi_k}(\cdot,w_j)$ in $B(z_0,\delta)$ for all large $j$, hence a contradiction. If $z_0\in\partial\Om$, we will work with $G_k(\cdot,a)$ on an appropriate neighborhood of $z_0$, and arrive at a contradiction by a similar reasoning. Thus, we have proved that $Z^i_k(a)$ are the only zeros of $S_{\varphi_k}(\cdot,a)$ and $L_{\varphi_k}(\cdot,a)$. Since $N_{ki}(a) = 1$, it follows that $S_{\varphi_k}(\cdot,a)$ and $L_{\varphi_k}(\cdot,a)$ have distinct zeros in $\Om$. 
		
		\medskip
		
		We have also proved that the combined number of zeros of $S_{\varphi_k}(\cdot, a)$ and $L_{\varphi_k}(\cdot,a)$ in $\Om$, plus the number of zeros of  $S_{\varphi_k}(\cdot,a)$ on $\partial\Om$ $($which are also the zeros of $L_{\varphi_k}(\cdot,a)$ on $\partial\Om)$, counting multiplicity, equals $n-1$, and are same as $Z_{k}^i(a)$. Moreover, these zeros are simple.
		
		\medskip
		
		Finally, let $\epsilon>0$ be given. Using (\ref{thms}), we can choose $k_1\ge 1$ with $k_1\geq k_0$ such that
		\[
		\left\lvert
		\int_{C_i} z \; \frac{{G'_k}(z, w)}{G_k(z,w)} dz - \int_{C_i}  z \; \frac{G'(z, w)}{G(z,w)} dz 
		\right\rvert
		\,\,<\,\,
		\frac{\epsilon}{2}
		\quad
		\text{for all }i=1,\ldots, n-1,
		\,w\in\overline{U}
		\text{ and }
		k\geq k_1.
		\]
		Choose $j_1 \ge 1$ with $j_1\geq j_0$ such that
		\[
		\left\lvert
		\int_{C_i} z \; \frac{G'(z, w_j)}{G(z,w_j)} dz - \int_{C_i} z \; \frac{G'(z, a)}{G(z,a)} dz 
		\right\rvert 
		\,\,<\,\,
		\frac{\epsilon}{2}
		\quad
		\text{for all }i=1,\ldots, n-1
		\text{ and }
		j\geq j_1.
		\]
		Since $Z^i(a)\in\partial\Omega$, we have $\vert Z^i_k(w_j) - Z^i(a) \vert = \vert Z_i(G_k(\cdot,w_j)) - Z^i(a)\vert$. Therefore, for $j\geq j_1$ and $k\geq k_1$, we obtain using the residue theorem that
		\begin{multline*}
			\vert Z^i_k(w_j) - Z^i(a) \vert
			=
			\vert Z_i(G_k(\cdot,w_j)) - Z^i(a)\vert
			=
			\frac{1}{2\pi}\left\lvert
			\int_{C_i} z \; \frac{G'_k(z, w_j)}{G_k(z,w_j)} dz - \int_{C_i} z \; \frac{G'(z, a)}{G(z,a)} dz
			\right\rvert
			\\
			\leq
			\frac{1}{2\pi}\left\lvert
			\int_{C_i} z \; \frac{G'_k(z, w_j)}{G_k(z,w_j)} dz - \int_{C_i} z \;\frac{G'(z, w_j)}{G(z,w_j)} dz
			\right\rvert
			+
			\frac{1}{2\pi}\left\lvert
			\int_{C_i} z \;\frac{G'(z, w_j)}{G(z,w_j)} dz - \int_{C_i} z \; \frac{G'(z, a)}{G(z,a)} dz
			\right\rvert
			\\
			<
			\frac{1}{2\pi }\left(\frac{\epsilon}{2} + \frac{\epsilon}{2} \right)
			= \frac{\epsilon}{2\pi}.
		\end{multline*}
		Thus, 
		\begin{equation}
		\lim_{j,k\rightarrow\infty} Z^i_{k}(w_j) = Z^i(a).
		\end{equation}
		This completes the proof of the theorem.		
	\end{proof}


	\section{Weighted Ahlfors maps}
	
	\begin{thm}[Nehari \cite{Nehari}]
		Let $\Omega\subset\mathbb{C}$ be a bounded $n$-connected domain with $C^{\infty}$-smooth boundary such that no boundary curve is an isolated point. Let $\varphi$ be a positive real-valued $C^{\infty}$-smooth function on $\partial\Om$. Given $a\in\Om$, let
		\[
		\mathcal{B}_{\varphi} = \left\lbrace
		f\in\mathcal{O}(\Om) : f(a) = 0 \text{ and } \limsup_{z\rightarrow z_0}\vert f(z)\vert \leq \frac{1}{\varphi(z_0)}\text{ for }z_0\in\partial\Om\right\rbrace.
		\]
		Then for every $f\in \mathcal{B}_{\varphi}$,
		\begin{equation}\label{ahlfors}
			\vert f'(a)\vert \leq \vert F'(a) \vert = 2\pi S_{\varphi}(a,a)
		\end{equation}
		where 
		\begin{equation}
		F(z) = \frac{S_{\varphi}(z,a)}{L_{\varphi}(z,a)}
		\end{equation}
		is the weighted Ahlfors map. If $L_{\varphi}(z,a) \neq 0$ for all $z\in \Om$, then $F\in \mathcal{B}_{\varphi}$ and it is the unique function (modulo rotation) in $\mathcal{B}_{\varphi}$ satisfying (\ref{ahlfors}).
	\end{thm}

		Since 
		\[
		\overline{S_{\varphi}(z,a)} = \frac{1}{i\varphi(z)} L_{\varphi}(z,a) T(z), \quad z\in\partial\Om,
		\]
		the $\lim_{z\rightarrow z_0}\vert F(z)\vert$ exists for all $z_0\in\partial\Om$, and
		\begin{equation}
		\lim_{z\rightarrow z_0} \vert F(z)\vert 
		=
		\frac{1}{\varphi(z_0)},\quad z_0\in\partial\Om.
		\end{equation}
	Thus, $F$ has a well defined limit on $\pa \Omega$ and this strengthens Nehari's theorem mentioned above.
	
	\begin{thm}
		Let $\Omega\subset\mathbb{C}$ be a bounded $n$-connected domain with $C^{\infty}$-smooth boundary such that no boundary curve is an isolated point. Let $\{\varphi_k\}_{k=1}^{\infty}$ be a sequence of positive real-valued $C^{\infty}$-smooth functions on $\partial\Omega$ such that $\varphi_k\rightarrow 1$ in the $C^{\infty}$ topology on $\partial\Omega$ as $k\rightarrow\infty$. For a compact set $W\subset\Omega$, there exists $k_0 \ge 1$ such that for all $a\in W$ and $k\geq k_0$
		\begin{itemize}
			\item[(i)] $L_{\varphi_k}(\cdot,a)$ does not vanish on $\overline{\Om}$ 
			\item[(ii)] $S_{\varphi_k}(\cdot,a)$ has $n-1$ zeros in $\Om$ and does not vanish on $\partial\Om$ 
		\end{itemize}
		Furthermore, the weighted Ahlfors maps $F_k$ with respect to $a\in W$ converge to the corresponding classical Ahlfors map $F$, uniformly on $\ov{\Om}$. The convergence is also uniform with respect to $a\in W$. In other words,
		\begin{equation}
		\lim_{k\rightarrow\infty} F_k(z,a)
		=
		\lim_{k\rightarrow\infty} \frac{S_{\varphi_k}(z,a)}{L_{\varphi_k}(z,a)} 
		=
		\frac{S(z,a)}{L(z,a)} 
		=:
		F(z,a)
		\end{equation}
		uniformly for $a\in W$ and $z\in\overline{\Om}$.
	\end{thm}
	
	\begin{proof}
		Choose $k_1 \ge 1$ such that $L_{\varphi_k}(z,a)$ and $S_{\varphi_k}(z,a)$ do not vanish for $k\geq k_1$, $z\in \partial\Om$ and $a\in W$. For $k\geq k_1$, 
		\[
		\frac{1}{i\varphi_k(z)}S_{\varphi_k}(z,a)L_{\varphi_k}(z,a) T(z) = \vert S_{\varphi_k}(z,a)\vert^2,
		\quad z\in\partial\Om, a\in\Om
		\]
		and this gives
		\[
		\Delta \arg S_{\varphi_k}(\cdot, a)L_{\varphi_k}(\cdot, a) = -\Delta \arg T = 2\pi (n-2)
		\]
		Since $L_{\varphi_k}(z,a)$ has a simple pole at $z=a$, the argument principle implies that the combined number of zeros of $S_{\varphi_k}(\cdot,a)$ and $L_{\varphi_k}(\cdot,a)$ in $\Om$ equals $n-1$, counting multiplicity, for all $a\in W$. The function
		\[
		N_k(a) = \frac{1}{2\pi i}\int_{\partial \Om}\frac{{S_{\varphi_k}}'(z,a)}{S_{\varphi_k}(z,a)} dz
		\]
		gives total number of zeros of $S_{\varphi_k}(\cdot,a)$ in $\Om$. The function $N_k$ converges uniformly for $a\in W$ to the function
		\[
		N(a) = \frac{1}{2\pi i} \int_{\partial \Om}\frac{S'(z,a)}{S(z,a)} dz
		\]
		that gives the number of zeros of $S(\cdot,a)$ in $\Om$, counting multiplicities, which is $n-1$. Therefore, eventually $N_k$ are constant functions. Choose $k_0\geq k_1$ such that $N_k(a) = n-1$ for all $k\geq k_0$ and $a\in W$. Thus, $S_{\varphi_k}(\cdot,a)$ has $n-1$ zeros in $\Om$ and does not vanish on $\partial\Om$ for all $k\geq k_0$ and $a\in W$. Since the combined number of zeros of $S_{\varphi_k}(\cdot,a)$ and $L_{\varphi_k}(\cdot,a)$ equals $n-1$, the function $L_{\varphi_k}(\cdot,a)$ does not vanish on $\overline{\Om}$ for all $k\geq k_0$ and $a\in W$.
		
		\medskip
		
		We know that
		\[
		\lim_{k\rightarrow\infty} S_{\varphi_k}(z,w) = S(z,w)
		\quad \text{and}\quad
		\lim_{k\rightarrow\infty} L_{\varphi_k}(z,w) = L(z,w)
		\]
		locally uniformly on $(\ov{\Om}\times\ov{\Om})\setminus \{(z,z): z\in\partial\Om\}$ and $(\ov{\Om}\times\ov{\Om})\setminus \{(z,z): z\in \overline{\Om}\}$ respectively. Therefore,
		\[
		\lim_{k\rightarrow\infty} F_k(z,a) = F(z,a)
		\]
		uniformly for $a\in W$ and $z\in\partial\Om$. Since $F_k(\cdot,a),F(\cdot,a)\in A^{\infty}(\Om)$ for all $k\geq k_0$ and $a\in W$, it follows by the maximum modulus principle that 
		\begin{eqnarray*}
			\sup\{ \vert F_k(z,a) - F(z,a)\vert : z \in \overline{\Om}, a\in W\} 
			&=&
			\sup\{ \vert F_k(z,a) - F(z,a)\vert : z \in \partial\Om, a\in W\}.
		\end{eqnarray*}
        This shows that $\lim_{k\rightarrow\infty} F_k(z,a) = F(z,a)$ uniformly for $a\in W$ and $z\in\overline{\Om}$.
	\end{proof}

    \section{Explicit Formulas for a class of weights}
	
	We start with two examples due to Nehari and Bell, giving more details.
	
	\begin{exam}[Nehari \cite{Nehari}]
		Let $\Omega\subset\mathbb{C}$ be a bounded $n$-connected domain with $C^{\infty}$-smooth boundary such that no boundary curve is an isolated point. For $a\in\Om$, let $a_1,\ldots a_r$ be the zeros of $S(z,a)$ with multiplicity $n_1,\ldots n_r$, respectively. For $1\leq k\leq r$ and $1\leq m_i\leq n_i$, let
		\[
		\varphi(z) = \vert (z - a_1)^{m_1} \ldots (z-a_k)^{m_k} \vert^2.
		\]
		Given $g\in H^2(\partial\Omega)$,
		\begin{eqnarray*}
			g(a) 
			&=&
			\frac{1}{\prod_{i=1}^k (a-a_i)^{m_i}} \, \left( g(a) \prod_{i=1}^k (a-a_i)^{m_i} \right)
			\\&=&
			\frac{1}{\prod_{i=1}^k (a-a_i)^{m_i}} 
			\int_{\partial\Om} g(\zeta) \prod_{i=1}^k (\zeta-a_i)^{m_i} \,\overline{S(\zeta,a)} \,ds
			\\&=&
			\frac{1}{\prod_{i=1}^k (a-a_i)^{m_i}} 
			\int_{\partial\Om} g(\zeta) \prod_{i=1}^k (\zeta-a_i)^{m_i} \,\overline{S(\zeta,a)} \frac{\varphi(\zeta)}{\prod_{i=1}^k \vert \zeta - a_i\vert^{2 m_i}}\,ds
			\\&=&
			\int_{\partial\Om} g(\zeta) \,\frac{\overline{S(\zeta,a)}}{\prod_{i=1}^k (\overline{\zeta} - \overline{a_i})^{m_i}(a-a_i)^{m_i}} \,\varphi(\zeta)\,ds.
		\end{eqnarray*}
		Therefore, by the uniqueness of the weighted Szeg\H{o} kernel function $S_{\varphi}(\cdot, a)$, we have
		\[
		S_{\varphi}(z,a) = \frac{S(z,a)}{\prod_{i=1}^k (z-a_i)^{m_i} (\overline{a}-\overline{a_i})^{m_i}}.
		\]
		Now for $z\in\partial\Omega$,
		\begin{eqnarray*}
			L_{\varphi}(z, a) 
			&=&
			i \varphi(z) \overline{T(z)} \overline{S_{\varphi}(z,a)} 
			\\&=&
			i \prod_{i=1}^{k} \vert z-a_i\vert^{2m_i} \,\overline{T(z)}  \frac{\overline{S(z,a)}}{\prod_{i=1}^k (\overline{z}-\overline{a_i})^{m_i} (a-a_i)^{m_i}}
			\\&=&
			\prod_{i=1}^{k}\frac{(z-a_i)^{m_i}}{(a-a_i)^{m_i}} \,L(z,a).
		\end{eqnarray*}
		The identity principle implies that
		\[
		L_{\varphi}(z, a) = \prod_{i=1}^{k}\frac{(z-a_i)^{m_i}}{(a-a_i)^{m_i}} \,L(z,a), \quad z\in\overline{\Om}.
		\]
		So, the weighted Garabedian kernel function $L_{\varphi}(\cdot,a)$ does not vanish on $\partial\Om$ and has zeros $a_1,\ldots, a_k$ in $\Om$ with multiplicity $m_1,\ldots, m_k$ respectively. 
		
		\medskip
		
		The weighted Ahlfors map $F_{\varphi}$ with respect to $a\in\Om$, is given by
		\[
		F_{\varphi}(z) = \frac{S_{\varphi}(z,a)}{L_{\varphi}(z,a)} 
		=
		\prod_{i=1}^k\frac{(a-a_i)^{2m_i}}{\vert a - a_i\vert^{2m_i}} \frac{1}{(z-a_i)^{2m_i}} 
		\frac{S(z,a)}{L(z,a)}.
		\]
		Thus, the weighted Ahlfors map $F_{\varphi}$ and the classical Ahlfors map $F$ with respect to $a\in \Om$ are related by
		\[
		F_{\varphi}(z) = \prod_{i=1}^k\frac{(a-a_i)^{2m_i}}{\vert a - a_i\vert^{2m_i}} \frac{1}{(z-a_i)^{2m_i}} 
		F(z).
		\] 
		If $2m_i > n_i$, then the function $F_{\varphi}$ has a pole at $a_i$ of order $2m_i-n_i$. If $2m_i\leq n_i$ for each $i$, then $F_{\varphi}\in \mathcal{B}_{\varphi}$.
		
		\medskip
		
		The function $F_{\varphi}$ maps the $n$ boundary curves $\gamma_i$ of $\Om$ onto $n$ Jordan closed curves $\Gamma_i$. Some of the images $\Gamma_i$ may coincide, for example all of them coincide if $\varphi\equiv 1$. The image of $\Om$ under $F_{\varphi}$ is open and connected and the boundary is the union of $\Gamma_i$. Since $F_{\varphi}$ has a pole, $F_{\varphi}(\Om)$ is the connected component of 
		\[
		(\mathbb{C}\cup\{\infty\} )\setminus \cup_{i=1}^n\Gamma_i
		\]
		that contains $\infty$.
	\end{exam}
	
	\begin{exam}[Bell \cite{Be1}]
		Let $\Omega\subset\mathbb{C}$ be a bounded $n$-connected domain with $C^{\infty}$-smooth boundary such that no boundary curve is an isolated point. Let $p(a,z)$ denote the Poisson kernel of $\Om$. For $A_0\in\Om$, let 
		\[
		\varphi(z) = p(A_0,z).
		\]
		Note that the weight $\varphi$ is well-defined for simply connected domains as well. For $g\in H^2(\partial\Om)$,
		\[
		g(A_0) = \int_{\partial\Om} g(\zeta) \varphi(\zeta) ds,
		\]
		by the reproducing property of the Poisson kernel. Therefore,
		\[
		S_{\varphi}(z,A_0) \equiv 1.
		\]
		Assume that $\Om$ is simply connected. Then,
		\[
		p(A_0,z) = \frac{\vert S(z,A_0)\vert^2}{S(A_0,A_0)}.
		\]
		Therefore, for $z\in\partial\Om$,
		\[
		L_{\varphi}(z,A_0) 
		=
		i \, \varphi(z) \, \overline{S_{\varphi}(z,A_0)} \, \overline{T(z)}
		=
		i \, \frac{\vert S(z,A_0)\vert^2}{S(A_0,A_0)} \, \overline{T(z)} = \frac{L(z,A_0) S(z,A_0)}{S(A_0,A_0)}.
		\]
		By the Identity principle,
		\[
		L_{\varphi}(z,A_0)  = \frac{L(z,A_0) \, S(z,A_0)}{S(A_0,A_0)}, \quad z\in\overline{\Om}.
		\]
		The weighted Ahlfors map $F_{\varphi}$ with respect to $A_0\in\Om$, is given by
		\[
		F_{\varphi}(z) = \frac{S(A_0,A_0)}{L(z,A_0) \, S(z,A_0)}.
		\]
		Note that the weighted Ahlfors map $F_{\varphi}\in \mathcal{B}_{\varphi}$ since $\Om$ is simply connected.
	\end{exam}
	
	\begin{exam}[see \cite{Nehari, zynda}]\label{mainexam}
		Let $\Omega\subset\mathbb{C}$ be a bounded $n$-connected domain with $C^{\infty}$-smooth boundary such that no boundary curve is an isolated point. Let 
		\[
		\varphi(z) = \vert f(z)\vert^2,
		\]
		where $f\in A^{\infty}(\Omega)$ is a non-vanishing function on $\overline{\Om}$.  Then
		\[
		S_{\varphi}(z,w) = \frac{1}{f(z)\overline{f(w)}} S(z,w)
		\quad\text{and}\quad
		L_{\varphi}(z,w)  = \frac{f(z)}{f(w)} L(z,w).
		\]
		Given $g\in H^2(\partial\Omega)$ and $a\in\Om$,
		\begin{eqnarray*}
			g(a) 
			&=&
			\frac{1}{f(a)} (gf)(a)
			=
			\frac{1}{f(a)} \int_{\partial\Om} (gf)(\zeta) \,\overline{S(\zeta,a)} \,ds
			=
			\frac{1}{f(a)} \int_{\partial\Om}(gf)(\zeta) \, \frac{\overline{S(\zeta,a)}}{\vert f(\zeta)\vert^2}\, \vert f(\zeta)\vert^2\, ds
			\\&=&
			\int_{\partial\Om}g(\zeta) \, \frac{\overline{S(\zeta,a)}}{\overline{f(\zeta)}f(a)}\, \varphi(\zeta)\, ds.
		\end{eqnarray*}
		Therefore, by the uniqueness of the weighted Szeg\H{o} kernel with respect to weight $\varphi$, we have
		\[
		S_{\varphi}(z,w) = \frac{1}{f(z)\overline{f(w)}} S(z,w),\quad z,w\in\Om.
		\]
		Now, for $z\in\partial\Om$ and $w\in\Om$,
		\[
		L_{\varphi}(z,w) 
		=
		i \,\vert f(z)\vert^2\, \overline{T(z)}\, \overline{S_{\varphi}(z,w)} 
		=
		i \,\vert f(z)\vert^2\, \overline{T(z)}\, \frac{1}{f(w)\overline{f(z)}}\,\overline{S(z,w)} 
		=
		\frac{f(z)}{f(w)} L(z,w).
		\]	
		The identity principle implies that
		\[
		L_{\varphi}(z,w)  = \frac{f(z)}{f(w)} L(z,w), 
		\quad z,w\in\Om.
		\]
		Thus, the weighted Szeg\H{o} kernel and the weighted Garabedian kernel inherit many nice properties of the classical Szeg\H{o} kernel and the Garabedian kernel, respectively.
		
		\medskip
		
		The weighted Ahlfors map $F_{\varphi}$ with respect to a point $a\in \Om$, is given by
		\[
		F_{\varphi}(z) = \frac{S_{\varphi}(z,a)}{L_{\varphi}(z,a)} 
		=
		\frac{f(a)^2}{\vert f(a)\vert^2} \, \frac{1}{f(z)^2} \frac{S(z,a)}{L(z,a)}.
		\]
		Thus, the weighted Ahlfors map $F_{\varphi}$ and the classical Ahlfors map $F$ with respect to a point $a\in\Om$, are related by
		\[
		F_{\varphi}(z) = \frac{f(a)^2}{\vert f(a)\vert^2} \, \frac{1}{f(z)^2} F(z).
		\]
		Here, $F_{\varphi}\in \mathcal{B}_{\varphi}$ for every $a\in\Om$, just as in the classical case.
		
		\medskip
		
		The function $F_{\varphi}$ maps the $n$ boundary curves $\gamma_i$ of $\Om$ onto $n$ Jordan closed curves $\Gamma_i$. Some of the images $\Gamma_i$ may coincide, for example all of them coincide if $\varphi\equiv 1$. The image of $\Om$ under $F_{\varphi}$ is open and connected and the boundary is the union of $\Gamma_i$. To see the image, $F_{\varphi}(\Om)$ is the connected component of 
		\[
		(\mathbb{C}\cup\{\infty\} )\setminus \cup_{i=1}^n\Gamma_i
		\]
		that contains $0$. The image is bounded as $F_{\varphi}$ has no poles.
	\end{exam}
	
	\begin{rem}
		Let $\Omega\subset\mathbb{C}$ be a bounded $n$-connected domain with $C^{\infty}$-smooth boundary such that no boundary curve is an isolated point. Let $\varphi$ be a positive real-valued $C^{\infty}$-smooth function on $\partial\Om$. It is known $($see \cite{UM, U-S}$)$ that the weighted Garabedian kernel $L_{\varphi}(z,w)$ does not vanish on $\overline{\Om}$ for all $w\in \Om$ if and only if there exists a non-vanishing function $f\in A^{\infty}(\Om)$ such that 
		\[
		L_{\varphi}(z,w) = L_{\vert f\vert^2}(z,w).
		\]
	\end{rem}
	
	\begin{exam}
		Let $\Omega\subset\mathbb{C}$ be a bounded $n$-connected domain with $C^{\infty}$-smooth boundary such that no boundary curve is an isolated point. For $A_0\in\Om$, let 
		\[
		\varphi(z) = \vert S(z,A_0)\vert^2.
		\]
		Given $g\in H^2(\partial\Omega)$,
		\begin{eqnarray*}
			g(A_0) 
			&=&
			\frac{1}{S(A_0,A_0)} \, g(A_0) \, S(A_0,A_0)
			=
			\frac{1}{S(A_0,A_0)} \int_{\partial\Om} g(\zeta) \, S(\zeta,A_0) \,\overline{S(\zeta,A_0)} \,ds
			\\&=&
			\frac{1}{S(A_0,A_0)} \int_{\partial\Om} g(\zeta) \,\varphi(\zeta) \,ds.
		\end{eqnarray*}
		Therefore, by the uniqueness of the weighted Szeg\H{o} kernel with respect to weight $\varphi$, we have
		\[
		S_{\varphi}(z,A_0) \equiv \frac{1}{S(A_0,A_0)}.
		\]
		Now, for $z\in\partial\Om$,
		\[
		L_{\varphi}(z,A_0) 
		=
		i \,\vert S(z,A_0)\vert^2\, \overline{T(z)}\, \overline{S_{\varphi}(z,A_0)} 
		=
		\frac{L(z,A_0)\,S(z,A_0)}{S(A_0,A_0)}.
		\]	
		The identity principle implies that
		\[
		L_{\varphi}(z,A_0)  = 	\frac{L(z,A_0)\,S(z,A_0)}{S(A_0,A_0)}, 
		\quad z\in\ov{\Om}.
		\]
		The weighted Ahlfors map $F_{\varphi}$ and the classical Ahlfors map $F$ with respect to $A_0\in \Om$, are related by
		\[
		F_{\varphi}(z) = \frac{1}{L(z,A_0)\,S(z,A_0)} = \frac{1}{S(z,A_0)^2} F(z).
		\]
        Here, $F_{\varphi}\in \mathcal{B}_{\varphi}$ if and only if $\Om$ is simply connected.
	\end{exam}
	
	\begin{exam}
		Let $\Omega\subset\mathbb{C}$ be a bounded $n$-connected domain with $C^{\infty}$-smooth boundary such that no boundary curve is an isolated point. For $A_0\in\Om$, let $B_0$ be a zero of $S(\cdot,A_0)$. Let 
		\[
		\varphi(z) = \frac{1}{\vert L(z,B_0)\vert^{2}}.
		\]
		Given $g\in H^2(\partial\Omega)$,
		\begin{eqnarray*}
			g(A_0) 
			&=&
			\frac{g(A_0)}{L(A_0,B_0)} \, L(A_0,B_0)
			=
			L(A_0,B_0) \int_{\partial\Om} \frac{g(\zeta)}{L(\zeta,B_0)} \,\overline{S(\zeta,A_0)} \,ds
			\\&=&
			L(A_0,B_0) \int_{\partial\Om}  \frac{g(\zeta)}{L(\zeta,B_0)} \,\overline{S(\zeta,A_0)} \,\vert L(\zeta,B_0)\vert^{2}\,\varphi(\zeta) \,ds
			\\&=&
			L(A_0,B_0) \int_{\partial\Om}  g(\zeta) \,\overline{S(\zeta,A_0)\,L(\zeta,B_0)} \,ds.
		\end{eqnarray*}
		Since $S(B_0,A_0) = 0$, the function $S(\cdot,A_0)\,L(\cdot,B_0) \in H^2(\partial\Om)$.
		Therefore, by the uniqueness of the weighted Szeg\H{o} kernel with respect to weight $\varphi$, we have
		\[
		S_{\varphi}(z,A_0) = S(z,A_0)\,L(z,B_0)\,\overline{L(A_0,B_0)}.
		\]
		Now, for $z\in\partial\Om$,
		\begin{eqnarray*}
			L_{\varphi}(z,A_0) 
			&=&
			i \, \frac{1}{\vert L(z,B_0)\vert^2} \, \overline{T(z)}\, \overline{S_{\varphi}(z,A_0)} 
			\\&=&
			i \, \frac{1}{\vert L(z,B_0)\vert^2} \, \overline{T(z)}\, \overline{S(z,A_0)}\,\overline{L(z,B_0)}\,L(A_0,B_0)
			\\&=&
			\frac{L(z,A_0)}{L(z,B_0)} \,L(A_0,B_0).
		\end{eqnarray*}
		The identity principle implies that
		\[
		L_{\varphi}(z,A_0)  = 	\frac{L(z,A_0)}{L(z,B_0)} \,L(A_0,B_0),
		\quad z\in\ov{\Om}.
		\]
		
		\medskip
		
		The weighted Ahlfors map $F_{\varphi}$ and the classical Ahlfors map $F$ with respect to $A_0\in \Om$, are related by
		\[
		F_{\varphi}(z) = \frac{\overline{L(A_0,B_0)}}{L(A_0,B_0)} \, L(z,B_0)^2 \, F(z).
		\]
		The weighted Ahlfors map $F_{\varphi}\in \mathcal{B}_{\varphi}$ if $S(z,A_0)$ vanishes at $z=B_0$ with multiplicity at least $2$. Otherwise, $F_{\varphi}$ has a simple pole at $z=B_0$.
	\end{exam}
	
	\begin{rem} We do not know an example of a weight $\varphi$ for which $L_{\varphi}(z,a)$ or $S_{\varphi}(z,a)$ vanishes for some $z\in\partial\Om$ and $a\in\Om$. The problem is to find an example of such a weight $\varphi$ or prove that such a weight cannot exist.
	\end{rem}
	
	
	\section{Nehari's Construction of the weighted Szeg\H{o} and weighted Garabedian kernels}

As before, let $\Om\subset\mathbb{C}$ be a bounded $n$-connected domain with $C^{\infty}$-smooth boundary curves $\gamma_i$ and no isolated boundary point. Let $\varphi$ be a positive $C^{\infty}$-smooth function on $\partial\Om$. 
Theorem $1$ in \cite[Nehari]{Nehari} shows that there exists a meromorphic function $q$ on $\Omega$ that extends $C^{\infty}$ smoothly to $\pa \Omega$ such that the total number of zeros and poles counting multiplicities does not exceed $n-1$, and $\vert q(z) \vert^2 = \varphi(z)$ on $\pa \Om$. The proof of this theorem is based on the construction of a suitable harmonic function together with a correction term involving a combination of Green's functions. Along the way, several auxiliary functions naturally appear. These functions play a more significant role than might first be apparent. They can be identified as the Szeg\H{o} kernels of certain closed subspaces of $H^2(\pa \Omega)$ and ultimately lead to explicit formulae for the weighted kernels associated with the weight $\varphi$.

\medskip

For the convenience of the reader, and because these auxiliary functions will be important in what follows, we briefly review Nehari's construction. Our goal is not merely to reproduce the proof of Theorem $1$ in \cite[Nehari]{Nehari}, but also to emphasize the origin and significance of the kernel functions that arise in the argument.

\medskip

For $a \in \Om$, recall that the Green's function $G(z, a)$ with pole at $z=a$ is defined as
\[
G(z,a) = -\log\vert z-a\vert + u_a(z)
\]
where $u_a(z)$ is the harmonic function of $z$ on $\Om$ that solves the Dirichlet problem with boundary data equal to $\log \vert z-a\vert$. The boundary weight $\varphi$ determines a harmonic function through the Poisson-type representation
\[
u(z) = -\frac{1}{4\pi} \int_{\partial\Om} \log \varphi(t) \frac{\partial G(z,t)}{\partial n}\, ds,
\]
which has boundary values $2^{-1}\log \varphi(t)$, $t \in \pa \Omega$. Let $v$ denote a harmonic conjugate of $u$. Since $\Omega$ is multiply connected, $v$ need not be single-valued.
Let $2\pi p_{\nu}$ denote the periods of the harmonic conjugate $v$ of $u$ about $\gamma_{\nu}$, $\nu = 1,\ldots,n$. The main idea now is to modify $u$ by adding to it an expression involving Green's functions in such a way that the periods of the resulting harmonic conjugate become integral multiples of $2 \pi$. 

\medskip

To this end, consider
\begin{equation}\label{Uform}
U(z) = u(z) + \sum_{\mu=1}^{n-1}\epsilon_{\mu} G(z,z_{\mu}),
\quad \quad z_{\mu}\in \Om,\quad \mu =1,\ldots,n-1,
\end{equation}
where $\epsilon_{\mu}$ is either $+1$ or $-1$. Since each Green's function vanishes on $\partial\Omega$, the function $U$ has exactly the same boundary values as $u$. Thus, the boundary condition determined by $\varphi$ is preserved, while the periods of the harmonic conjugate may be altered. 

\medskip

To understand the effect of this modification, let $V$ be a harmonic conjugate of $U$. The period of $G(z,\z_{\mu})$ about $\gamma_{\nu}$ is
\[
\int_{\gamma_{\nu}} \frac{\partial G(z,z_{\mu})}{\partial n} ds 
=
(-2\pi) \frac{-1}{2\pi} \int_{\gamma_{\nu}} \frac{\partial G(z,z_{\mu})}{\partial n} ds
=
(-2\pi) \,\omega_{\nu}(z_{\mu})
\]
where $\omega_{\nu}$ denotes the harmonic function on $\Omega$ that takes the boundary values $\delta_{\nu\mu}$ on $\gamma_{\mu}$. Therefore, the period $2\pi P_{\nu}$ of the harmonic conjugate $V$ of $U$ about $\gamma_{\nu}$ is given by
\[
P_{\nu} = p_{\nu} - \sum_{\mu=1}^{n-1} \epsilon_{\mu} \omega_{\nu}(z_{\mu}), 
\quad \quad \nu =1,\ldots,n.
\]
We want to choose $z_{\mu}$ and $\epsilon_{\mu}$ in such a way that $P_{\nu}$ are integers. In that case the multivalued factor arising from the harmonic conjugate will disappear after exponentiation.

\medskip

Since $u$ is harmonic, its periods satisfy
\[
p_1+\ldots+p_n = 0
\]
and likewise, the harmonic measures satisfy
\[
\omega_1+\ldots+\omega_n \equiv 1.
\]
Using these expressions, we get
\[
P_n = -\sum_{\nu=1}^{n-1} \left(p_{\nu} 
-
\sum_{\mu=1}^{n-1}\epsilon_{\mu}\omega_{\nu}(z_{\mu})\right) -\sum_{\mu=1}^{n-1}\epsilon_{\mu}
\]
and therefore, it is enough to ensure that $P_{\nu}$ are integers for $\nu=1,\ldots,n-1$. That is, we have to find $n-1$ points $z_1,\ldots,z_{n-1}$ in $\Om$ for which 
\begin{equation}\label{Ne1}
	\sum_{\mu=1}^{n-1}\epsilon_{\mu}\omega_{\nu}(z_{\mu}) = p_{\nu} + m_{\nu},
	\quad\quad \nu =1,\ldots, n-1
\end{equation}
holds for appropriately chosen integers $m_{\nu}$ and for a suitable choice of $\epsilon_{\mu}$. Since harmonic measure is conformally invariant, a schlicht conformal map carries the harmonic measures of $\Omega$ to those of its image domain. We may therefore assume that the boundary component $\gamma_n$ coincides with a segment of the real axis and that $\Omega$ lies in the upper half-plane. Since the harmonic functions $\omega_{\nu}$ vanish on the real axis, they can be analytically continued beyond the real axis by the identity $\omega_{\nu}(\ov{z}) = - \omega_{\nu}(z)$. The functions $\omega_{\nu}$ are thus harmonic in the entire domain $\Omega' = \Omega \cup \mathbb{R} \cup \widetilde{\Om}$, where $\widetilde{\Om}$ is the mirror image of $\Om$ with respect to the real axis. Since $\omega_{\nu}(\ov{z}) = - \omega_{\nu}(z)$, the system of equations (\ref{Ne1}) is equivalent to the system
\begin{equation}\label{Ne0}
	\sum_{\mu=1}^{n-1} \omega_{\nu}(z_{\mu}) = p_{\nu} + m_{\nu}, 
	\quad \quad \nu = 1,\ldots,n-1,
\end{equation}
if $z_{\mu}$ are not restricted to $\Om$ and can take any value in $\Omega'$. Now, consider the expression
\begin{equation}\label{Ne2}
	R = R(z_1,\ldots,z_{n-1})
	=
	\sum_{\nu =1}^{n-1}\left[\sum_{\mu=1}^{n-1} \omega_{\nu}(z_{\mu}) - p_{\nu} - m_{\nu}\right]^2.
\end{equation}
The function $R(z_1,\ldots,z_{n-1})$ is continuous with variables coming from $\Omega' \cup \Gamma'$ where $\Gamma'$ denotes the boundary of $\Omega'$. Thus, for a definite choice of integers $m_{\nu}$, $R$ has a non-negative minimum, which is attained for a set of $n-1$ points $z_{\mu}$ in $\Omega' \cup \Gamma'$. We now choose $m_{\nu}$ in such a way as to give this minimum its smallest possible value, which is possible because $-1\leq \omega_{\nu}(z)\leq 1$. If $m_1,\ldots,m_{n-1}$ are a choice of integers which give the smallest minimum of $R$, we write $p_{\nu} + m_{\nu} = \alpha_{\nu}$, and consider the solution of
\[
R = R(z_1,\ldots,z_{n-1}) 
=
\sum_{\nu =1}^{n-1}\left[\sum_{\mu=1}^{n-1} \omega_{\nu}(z_{\mu}) - \alpha_{\nu} \right]^2 = {\rm min},
\]
where $z_{\mu} \in \Omega'\cup \Gamma'$. Suppose $z_1\in \gamma_k$ ($k\neq n$). Since $\omega_k(z_1) = 1$, $\omega_{\nu}(z_1) = 0$ $(\nu\neq k)$, and $\omega_{\nu}(z)$ ($\nu =1,\ldots,n-1$) vanishes on the real axis, (\ref{Ne2}) will remain same if $z_1$ is transferred to the real axis and $m_k$ is replaced by $m_k-1$. Similarly, a point $z_{\mu}$ may be transferred from the mirror image of $\gamma_k$ to the real axis keeping (\ref{Ne2}) unchanged by replacing $m_k$ with $m_k + 1$. Therefore, there exists a minimizing set $S = \{z_{\mu}\}$ lying entirely in $\Om'$. Similarly, we can assume that $S$ does not contain both a point $z_{\mu}$ and its conjugate $\ov{z}_{\mu}$, because they can be replaced by points on the real axis without altering $(\ref{Ne2})$. Since $z_{\mu}$ minimize (\ref{Ne2}), we have
\[
0 =  \sum_{\nu =1}^{n-1}\left[\sum_{\mu=1}^{n-1} \omega_{\nu}(z_{\mu})-\alpha_{\nu}\right] {w'_{\nu}}(z_{\mu}) 
=
\sum_{\nu =1}^{n-1} a_{\nu} {w'_{\nu}}(z_{\mu}),
\quad \quad
a_{\nu} = \sum_{\mu=1}^{n-1} \omega_{\nu}(z_{\mu})-\alpha_{\nu}
\]
where $w_{\nu}(z)$ is (locally defined) holomorphic function whose real part is $\omega_{\nu}(z)$. Note that the derivative of $w_{\nu}$ is globally defined and equals $(\partial/\partial z) \omega_{\nu}$. 

\medskip

Now, either all $a_{\nu}$ vanish or the derivative of the (locally defined) holomorphic function
\[
w(z) = \sum_{\nu=1}^{n-1} a_{\nu} w_{\nu}(z)
\] 
vanishes at $n-1$ points $z_1,\ldots,z_{n-1}$. In the first case, we have solved (\ref{Ne0}). So, assume that at least one of the constants $a_{\nu}$ is non-zero. The points of $S$ are critical points of $w(z)$. Since $a_{\nu}$ are real, we have ${w'_{\nu}}(\ov{z}) = - \overline{{w'_{\nu}}(z)}$. So, the set $S^*$ of points conjugate to those of $S$ is also a set of critical points of $w(z)$. If no point of $S$ is on the real axis, the sets $S$ and $S^*$ are disjoint and we thus have $2(n-1)$ critical points of $w(z)$, and $n-1$ of these points are in $\Om$. It is known that
\[
\overline{i \,\frac{\partial \omega_{\nu}}{\partial z}\, T(z)} = i \,\frac{\partial \omega_{\nu}}{\partial z} \,T(z)
\]
on $\partial\Omega$. Therefore, $i w'(z) T(z)$ is real on $\partial\Om$. By the argument principle, we have
\begin{eqnarray*}
 \text{no. of zeros of } w'(z) \text{ in } \Om &=& \frac{1}{2\pi} \Delta \arg w'(z) 
=
\frac{1}{2\pi} \left(\Delta\arg (i w'(z) T(z)) - \Delta\arg T(z)\right) 
\\&=& n-2,
\end{eqnarray*}
which is a contradiction. So, assume that one or more points of $S$ lie on the real axis. Note that $R$ is constant on the real axis. If $z_{\mu}$ is a point of $S$ lying on the real axis, then it can be replaced with any other point on the real axis without altering the value of the minimum. So, all the points of the real axis are critical points of $w(z)$. This means that
\[
\sum_{\nu=1}^{n-1} a_{\nu} {w'_{\nu}}(z) \equiv 0
\]
on the real axis, and by analytic continuation, throughout $\Om'$. But this is possible only if all $a_{\nu}$ vanish because ${w'_{\nu}}(z)$ are linearly independent. This leads to a contradiction and shows the solvability of (\ref{Ne0}).

\medskip

Therefore, the periods of the harmonic conjugate $V$ of $U$ are integral multiples of $2\pi$. Thus,
\[
q(z) = \exp\{U(z) + i V(z)\}
\]
is a single-valued meromorphic function on $\Om$ with simple zeros or simple poles at $z_{\mu}$, depending on whether $\epsilon_{\mu}$ is equal to $+1$ or $-1$. The number of zeros and poles of $q(z)$ counting multiplicity, does not exceed $n-1$ , and
\[
\vert q(z)\vert ^2 = \exp(2\,U(z)) = \exp\left(\log \varphi(z)\right) = \varphi(z),
\quad \quad z\in \partial\Omega.
\]
Note that if two of the points $z_{\mu}$ coincide but the respective $\epsilon_{\mu}$ are opposite in sign then the corresponding terms involving Green's function in (\ref{Uform}) get cancelled and the total number of zeros and poles of $q$ counting multiplicity will be less than $n-1$.

\medskip

Now, for $\zeta\in\Om$, consider the functions
\begin{equation}\label{Neq1}
	S^{(1)}(z,\zeta) = S(z,\zeta) + \sum_{\nu=1}^\sigma c_{\nu} S(z,z_{\nu}) + \sum_{\nu = \sigma +1}^{n-1} c_{\nu} L(z,z_{\nu}),
	\quad\quad \sigma\leq n-1,
\end{equation}
\begin{equation}\label{Neq2}
	L^{(1)}(z,\zeta) = L(z,\zeta) + \sum_{\nu=1}^\sigma \ov{c}_{\nu} L(z,z_{\nu}) + \sum_{\nu = \sigma +1}^{n-1} \ov{c}_{\nu} S(z,z_{\nu})
\end{equation}
where $z_1,\ldots,z_{\sigma}$ are the zeros of $q(z)$ and $z_{\sigma+1},\ldots,z_{n-1}$ are the poles of $q(z)$. The $n-1$ constants $c_{\nu}$ shall be determined by
\begin{equation}\label{Neq3}
	S^{(1)}(z_{\mu},\zeta) = 0, \quad \mu =1,\ldots, \sigma;
	\quad\quad
	L^{(1)}(z_{\mu}, \zeta) = 0, \quad \mu = \sigma+1,\ldots,n-1.
\end{equation}
This will be possible if the system of $n-1$ linear equations
\begin{multline}\label{Ne3}
	\sum_{\nu=1}^\sigma c_{\nu} S(z_{\mu},z_{\nu}) + \sum_{\nu = \sigma +1}^{n-1} c_{\nu} L(z_{\mu},z_{\nu})
	=
	-S(z_{\mu},\zeta),
	\quad\quad \mu = 1,\ldots,\sigma,
	\\
	\sum_{\nu=1}^\sigma c_{\nu} \overline{L(z_{\mu},z_{\nu})} + \sum_{\nu = \sigma +1}^{n-1} c_{\nu} \overline{S(z_{\mu},z_{\nu})} 
	=
	-\overline{L(z_{\mu},\zeta)},
	\quad\quad \mu = \sigma+1,\ldots,n-1,
\end{multline}
for the $n-1$ unknowns $c_1,\ldots,c_{n-1}$ has a solution. We may assume that not all the numbers $S(z_{\mu}, \zeta)$ $(\mu =1,\ldots,\sigma)$, $L(z_{\mu},\zeta)$ $(\mu = \sigma+1,\ldots,n-1)$ are zero. It is enough to show that the associated homogeneous system in $(\ref{Ne3})$ does not possess a non-trivial solution. On the contrary, if it had a non-trivial solution, there would exist constants $d_{\nu}$ $(\nu=1,\ldots, n-1)$, not all of which are zero, such that
\[
S^{(2)}(z_{\mu}) = 0,\quad \mu =1,\ldots, \sigma;\quad\quad
L^{(2)}(z_{\mu}) = 0, \quad \mu = \sigma+1,\ldots,n-1,
\]
where
\[
S^{(2)}(z) = \sum_{\nu =1}^{\sigma} d_{\nu} S(z,z_{\nu}) + \sum_{\nu = \sigma+1}^{n-1} d_{\nu} L(z,z_{\nu}),
\quad
L^{(2)}(z) = \sum_{\nu=1}^{\sigma} \ov{d}_{\nu} L(z,z_{\nu}) + \sum_{\nu=\sigma +1}^{n-1} \ov{d}_{\nu} S(z,z_{\nu}).
\]
If two of the points $z_{\nu}$ coincide, at the point $z_{\tau}$, $\tau\leq \sigma$, then (\ref{Neq1}) and (\ref{Neq2}) will each have an extra term involving $\partial S(z,z_{\tau})/\partial \ov{z}_{\tau}$ and $\partial L(z,z_{\tau})/\partial z_{\tau}$, respectively. The condition in (\ref{Neq3}) which is lost by the two points coinciding is replaced by the condition $\partial S^{(1)}(z_{\tau}, \zeta)/\partial z_{\tau} = 0$. If $\tau > \sigma$, then the roles of the functions $S$ and $L$ are interchanged in modifying (\ref{Neq1}), (\ref{Neq2}) and the condition (\ref{Neq3}). Coincidence of more than two points will result in the appearance of higher derivatives in the appropriate places. These changes do not affect the validity of the considerations which follow. We have
\[
\overline{S^{(2)}(z)} = \frac{1}{i} L^{(2)}(z) T(z),
\quad \quad z\in \partial \Om.
\]
Therefore,
\[
\frac{1}{i} S^{(2)}(z)L^{(2)}(z) T(z) \geq 0,
\quad\quad z\in\partial\Omega,
\]
which implies that
\[
\frac{1}{i} \int_{\partial\Om}S^{(2)}(z)L^{(2)}(z) dz >0. 
\]
By Cauchy's theorem, the function $S^{(2)}(z)L^{(2)}(z)$ cannot be analytic in $\Om$, which is a contradiction. Therefore, we have proved that ($\ref{Ne3}$) has a unique solution. Furthermore,
\begin{equation}\label{S1L1}
\overline{S^{(1)}(z,\zeta)} = \frac{1}{i} L^{(1)}(z,\zeta) T(z),
\quad\quad z\in \partial\Om.
\end{equation}
The weighted Szeg\H{o} and Garabedian kernels are given by
\begin{equation}
	S_{\varphi}(z,\zeta) = \frac{S^{(1)}(z,\zeta)}{q(z)\, \overline{q(\zeta)}},
	\quad\quad
	L_{\varphi}(z,\zeta) = \frac{L^{(1)}(z,\zeta) \,q(z)}{q(\zeta)}.
\end{equation}
If $q(z)$ happens to have a zero or a pole at $z=\zeta$, the above definitions of $S_{\varphi}(z,\zeta)$ and $L_{\varphi}(z,\zeta)$ need to be slightly modified. If $\eta$ is a point close to $\zeta$, we define
\begin{equation}
S_{\varphi}(z,\zeta) = \lim_{\eta\rightarrow \zeta} S_{\varphi}(z,\eta),
\quad\quad 
L_{\varphi}(z,\zeta) = \lim_{\eta\rightarrow\zeta} L_{\varphi}(z,\eta).
\end{equation}
This completes the review of Nehari's theorem. In particular, it provides a concrete description of the weighted Szeg\H{o} and Garabedian kernels in terms of the auxiliary functions $S^{(1)}(z, \zeta), L^{(1)}(z, \zeta)$ whose construction is ultimately rooted in the solution of the period-adjustment problem for the harmonic function determined by the boundary weight $\varphi$.

\medskip

We shall now see that the function $q$ in Nehari's construction can be chosen to have no poles, that is, $q\in A^{\infty}(\Om)$ with at most $n-1$ zeros in $\Om$ counting multiplicities. We need the following result due to Černe.

\begin{thm}[{\cite[Černe, Theorem 1.1]{Miran}}]\label{Miran}
Let $\Sigma$ be a bordered Riemann surface with genus $g$ and $m$ real analytic boundary components. Let $\{\gamma_z\}_{z \in \partial\Sigma}$ be a
$C^{k+1}$ ($k \geq 3$) family of Jordan curves
in $\mathbb{C}$ which all contain the point $0$
in their interior.
Then there exists a holomorphic function $f \in A^{k,\alpha}(\Sigma)$
with at most $2g+m-1$ zeros on $\Sigma$ such that
\[
f(z) \in \gamma_z
\qquad \text{for every } z \in \partial\Sigma.
\]
\end{thm}

Here, a family of simple closed curves $\{\gamma_z\}_{z \in \partial\Sigma}$ in $\mathbb{C}$
is a $C^k$ ($k \in \mathbb{N}$) family of Jordan curves if there exists a real-valued function
$\rho \in C^k(\partial\Sigma \times \mathbb{C})$
such that
\[
\gamma_z
=
\{w \in \mathbb{C}: \ \rho(z,w)=0\}
\]
and
$\frac{\partial\rho}{\partial\overline{w}}(z,w) \neq 0$ for every $z \in \partial\Sigma$ and $w \in \gamma_z$.
Moreover, $0 < \alpha < 1$, $k$ is a non-negative integer, $C^{k,\alpha}(\ov\Sigma)$ denotes the H\"older space of all $k$-times differentiable functions on $\ov\Sigma$ whose derivatives of order $k$ are H\"older continuous of order $\alpha$, and $A^{k,\alpha}(\Sigma)$ denotes the space of holomorphic functions on $\Sigma$ which are of class $C^{k,\alpha}$ on $\ov\Sigma$.

\begin{prop}\label{q_phi}
Let $\Om\subset\mathbb{C}$ be a bounded $n$-connected domain with $C^{\infty}$-smooth boundary and no isolated boundary point. Let $\varphi$ be a positive $C^{\infty}$-smooth function on $\partial\Om$. There exists $q\in A^{\infty}(\Om)$ with at most $n-1$ zeros in $\Om$ counted with multiplicity, such that
\begin{equation}
\vert q(z)\vert^2 = \varphi(z),
\qquad z\in\partial\Om.
\end{equation}
\end{prop}

\begin{proof}
There exists an $n$-connected bounded domain $\Om_0$ with real analytic boundary and no isolated boundary point, and a biholomorphic map $F:\Om\rightarrow\Om_0$. Furthermore, $F$ extends $C^{\infty}$ smoothly to $\partial\Om$ and its derivative is non-vanishing on $\ov{\Om}$; see \cite[Lemma 12.1]{Be0} for the proof.

\medskip

Therefore, proving the proposition for $\Om_0$ with weight function $\varphi_0=\varphi\circ F^{-1}$ gives $q_0\in A^{\infty}(\Om_0)$ with at most $n-1$ zeros in $\Om_0$ counted with multiplicity, such that $\vert q_0(z)\vert^2 = \varphi_0(z)$ for $z\in\partial\Om_0$. Then, taking $q= q_0\circ F$ proves the proposition for $\Om$. Hence, we may assume, without loss of generality, that $\Om$ has real analytic boundary.

\medskip

So, $\Om$ is a bordered Riemann surface with genus $0$ and $n$ real analytic boundary components. For each $z\in b\Omega$, consider the circle
\[
\gamma_z
=
\left\{
w\in\mathbb C: |w|^2=\varphi(z)
\right\}.
\]
Since $\varphi\in C^\infty(b\Omega)$ is strictly positive, the family $\{\gamma_z\}_{z\in b\Omega}$ is a smooth family of Jordan curves, each containing the point $0$ in its interior.
Indeed, the family admits the defining function
\[
\rho(z,w)=|w|^2-\varphi(z),
\]
which satisfies $\frac{\partial\rho}{\partial\overline{w}}(z,w) =w\neq 0$ for every $z\in b\Omega$ and $w\in\gamma_z$. For each $k\geq 3$, the family $\{\gamma_z\}_{z\in b\Omega}$ is $C^{k+1}$. 

\medskip

Therefore, by Theorem \ref{Miran}, fixing $k\geq 3$, there exists a holomorphic function $q$ on $\Omega$, such that $q\in C^{k,\alpha}(\ov{\Om})$. Moreover, $q$ has at most $n-1$ zeros in $\Om$ counting multiplicities, such that $q(z)\in\gamma_z$ for every $z\in\partial\Om$, that is,
\[
\vert q(z)\vert^2 = \varphi(z)
\qquad
\text{for each }z\in\partial\Om.
\]
It remains to verify that $q\in A^\infty(\Omega)$. Fix $p\in\partial\Om$ and choose a sufficiently small neighborhood $U$ of $p$ such that $U\cap\Om$ is simply connected and $q$ is
nonvanishing there. Hence, we may choose a
holomorphic branch
\[
\log q=u+iv
\]
on $U\cap\Om$. The function $u=\log|q|$ is harmonic on
$U\cap\Om$ and extends continuously to the boundary $\partial\Om$ near $p$ satisfying
\[
u=\frac12\log\varphi
\qquad\text{on }U\cap\partial\Om.
\]
Since $\partial\Om$ is $C^{\infty}$-smooth and $\log\varphi\in C^\infty(\partial\Om)$, local boundary regularity for the Dirichlet problem implies that $u$ is $C^\infty$ smooth up to the boundary $\partial\Om$ near $p$. By the Cauchy--Riemann equations, the local
harmonic conjugate $v$ is also $C^\infty$-smooth up to the boundary $\partial\Om$ near $p$. Consequently, $q=e^{u+iv}$ is smooth up to $\partial\Om$
near $p$.
Since $p\in\partial\Om$ was arbitrary, we conclude that $q\in A^\infty(\Om)$. This finishes the proof.
\end{proof}

As an immediate consequence of Proposition \ref{q_phi}, we recover the following known formulas for the weighted Szeg\H{o} and Garabedian kernels of smoothly bounded simply connected domains. These formulas give explicit expressions of the weighted kernels once the Riemann map of \(\Omega\) and the corresponding function \(q\) are known. Proposition \ref{q_phi} also answers the first part of the problem stated in \cite{AJ}.

\begin{cor}[see \cite{Nehari, UM}]\label{SC1}
Let $\Om$ be a bounded simply connected domain with $C^{\infty}$-smooth boundary and $\varphi$ be a positive $C^{\infty}$-smooth function on $\partial\Om$. Then there exists a non-vanishing $q\in A^{\infty}(\Om)$ such that
\begin{equation}
S_{\varphi}(z,w) = \frac{S(z,w)}{q(z)\,\overline{q(w)}},
\quad\quad \text{and} \quad\quad
L_{\varphi}(z,w) = \frac{L(z,w) \,q(z)}{q(w)}.
\end{equation}
\end{cor}

\begin{proof}
If $\Omega$ is simply connected, then $n=1$ in Proposition \ref{q_phi}. Therefore, there exists $q\in A^{\infty}(\Om)$ such that $\vert q(z)\vert^2 = \varphi(z)$ for each $z\in\partial\Om$. Moreover, $q$ does not have any zero in $\Omega$. Therefore, the corollary follows from Example \ref{mainexam}.
\end{proof}

\begin{thm}\label{neh}
Let $\Om$ be a bounded $n$-connected domain with $C^{\infty}$-smooth boundary and no isolated boundary point. Let $\varphi$ be a positive $C^{\infty}$-smooth function on $\partial\Om$. There exists $q\in A^{\infty}(\Om)$ with at most $n-1$ zeros counting multiplicities such that
\[
\vert q(z)\vert^2 = \varphi(z),
\quad\quad z\in\partial\Om.
\]
Let $z_1,\ldots,z_{\sigma}$ be the zeros of $q$ with multiplicities $m_1,\ldots,m_{\sigma}$ respectively, and $A=\sum_{i=1}^{\sigma}m_i\cdot z_i$ be the effective divisor. Let $S_{A}(z,w)$ denote the reproducing kernel of the Hilbert space
\[
H_{A}^2(\Omega) = \{f\in H^2(\partial\Omega): f(z_{\mu}) = 0 \text{ with multiplicity }m_{\mu}, \,\,\mu = 1,\ldots,\sigma\}.
\]
Then, the weighted Szeg\H{o} kernel $S_{\varphi}$ is given by
\[
S_{\varphi}(z,w) = \frac{S_A(z,w)}{q(z)\ov{q(w)}}
\quad \text{ if }w\neq z_\mu,
\quad\quad
S_{\varphi}(z,z_{\mu}) = \lim\limits_{\eta\rightarrow z_{\mu}}  S_{\varphi}(z,\eta).
\]
The orthogonal projection of the Cauchy kernel $C_a$ onto $H_A^2(\Om)$ is $S_A(\cdot,a)$. And, the orthogonal decomposition of $C_a$ satisfies
\[
C_a \ov{q}= S_A(\cdot, a) \ov{q} + \overline{TH_a}, 
\quad\quad \text{for some }H_a\in H^2(\partial\Om).
\]
We define the function $L_A$ by
\[
L_A(\zeta,a) = i \left(\frac{1}{2\pi i}\frac{1}{\zeta - a}- \frac{H_a(\zeta)}{q(\zeta)}\right).
\]
The functions $S_A$ and $L_A$ are related by
\[
\overline{S_A(\zeta,a)} = \frac{1}{i} L_A(\zeta,a) T(\zeta),
\quad \quad \zeta\in\partial\Om, a\in\Om.
\]
The weighted Garabedian kernel $L_{\varphi}$ is given by
\[
L_{\varphi}(z,w) = \frac{L_A(z,w) q(z)}{q(w)}
\quad \text{ if } w\neq z_{\mu},
\quad\quad
L_{\varphi}(z,z_{\mu}) = \lim\limits_{\eta\rightarrow z_{\mu}} L_{\varphi}(z,\eta).
\]
\end{thm}

\begin{proof}
The space $H_{A}^2(\Omega)$ is a closed subspace of $H^2(\partial\Om)$ and therefore it is a Hilbert space. For $w\in\Om$, the evaluation functionals 
\[
E_w: H_{A}^2(\Om) \ni f \mapsto f(w) \in \mathbb{C}
\]
are continuous because they are continuous on $H^2(\partial\Om)$. Therefore, by the Riesz Representation theorem, there exists a unique function $S_{A}(\cdot,w)\in H_{A}^2(\Om)$ such that
\begin{equation}
f(w) = \int_{z\in\partial\Om} f(z) \,\overline{S_{A}(z,w)}\, ds.
\end{equation}
The function $S_{A}(z,w)$ is the reproducing kernel of $H_{A}^2(\Om)$. For $f\in H^2_{\varphi}(\partial\Om)$, the function $f.q \in H_{A}(\Om)$. Therefore, for $w\in\Om$,
\[
f(w) q(w) = \int_{z\in\partial\Om}  f(z)\,q(z) \,\overline{S_{A}(z,w)}\,ds 
=
\int_{z\in\partial\Om}  f(z)\,\frac{\overline{S_{A}(z,w)}}{\overline{q(z)}}\,\varphi(z)\,ds.
\]
If $q(w)\neq 0$, it follows by the uniqueness of the weighted Szeg\H{o} kernel that
\begin{equation}
S_{\varphi}(z,w) = \frac{S_{A}(z,w)}{q(z)\,\overline{q(w)}}.
\end{equation}
In the case $q(w) = 0$, if $\eta$ is a point close to $w$, we have
\begin{equation}
S_{\varphi}(z,w) = \lim\limits_{\eta\rightarrow w} S_{\varphi}(z,\eta) = 
 \lim\limits_{\eta\rightarrow w} \frac{S_{A}(z,\eta)}{q(z)\,\overline{q(\eta)}}.
\end{equation}
Thus, the function $S^{(1)}(z,w)$ in Nehari's construction is the function $S_{A}(z,w)$.

\medskip

Now, let $v\in H_A^2(\Om)^{\perp}$. For $u\in L^2(\partial\Om)$, the function $q\,\mathcal{C}u \in H_A^2(\Omega)$. Therefore,
\[
0=\langle q\,\mathcal{C}u, v\rangle
= \langle \mathcal{C}u, v\ov{q}\rangle
= \langle u, \mathcal{C}^* (v\ov{q})\rangle
\]
This implies that $ \mathcal{C}^* (v\ov{q}) = 0$, that is, $v\ov{q} = \ov{T}\, \overline{\mathcal{C}(\ov{v}q\ov{T})}$. Thus, $v \ov{q} = \overline{TH}$ for some $H\in H^2(\partial\Om)$. Conversely, let $v\in L^2(\partial\Om)$ be such that $v \ov{q} = \overline{TH}$ for some $H\in H^2(\partial\Om)$. Then, for any $u\in H_A^2(\Om)$
\[
\langle u, v \rangle = \left\langle \frac{u}{q}, v\ov{q} \right\rangle 
=  \left\langle \frac{u}{q}, \overline{TH}\right\rangle = 0
\]
because $u/q\in H^2(\partial\Omega)$. Thus, every $v\in H^2_A(\Om)^{\perp}$ satisfies $v\ov{q} = \overline{TH}$ for some $H\in H^2(\partial\Om)$. Let $P_A$ denote the orthogonal projection of $L^2(\partial\Om)$ onto $H^2_A(\partial\Om)$. The orthogonal decomposition of $u\in L^2(\partial\Om)$ satisfies
\[
u\, \ov{q} = (P_A u) \,\ov{q} + \overline{TH}
\]
where $H = P_A(\ov{u}q\overline{T})$ because $\ov{u}\,q\,\overline{T} = H + \overline{(P_A u)}\, q\, \overline{T}$ is the orthogonal decomposition of $\ov{u}q\overline{T}$.

\medskip

Now, the orthogonal projection of the Cauchy kernel $C_a$ onto $H_A^2(\Om)$ is $S_A(\cdot,a)$ because
\[
P_A C_a (\zeta) = \int_{z\in\partial \Om} C_a(z) S_A(\zeta,z) ds = S_A(\zeta,a)
\]
by the Cauchy integral formula. Therefore, writing the orthogonal decomposition of $C_a$ gives
\[
\overline{\left(\frac{1}{2\pi i} \frac{T(\zeta)}{\zeta - a}\right)}\, \ov{q(\zeta)}
=
C_a(\zeta)\, \ov{q(\zeta)}
= S_A(\zeta,a)\, \ov{q(\zeta)} + \overline{T(\zeta) H_a(\zeta)}
\]
where $H_a\in H^2(\partial\Om)$ and $H_a = P_A(\ov{C_a} q \ov{T})$. We now define
\begin{equation}
L_A(\zeta,a) = i \left( \frac{1}{2\pi i} \frac{1}{\zeta - a} - \frac{H_a(\zeta)}{q(\zeta)}\right).
\end{equation}
By the definition, it follows that for $\zeta\in\partial\Om$,
\begin{eqnarray*}
\overline{S_A(\zeta,a)} = \left( \frac{1}{2\pi i} \frac{1}{\zeta - a} - \frac{H_a(\zeta)}{q(\zeta)} \right) T(\zeta) = \frac{1}{i} L_A(\zeta,a) T(\zeta).
\end{eqnarray*}
That is,
\begin{equation}
\overline{S_A(\zeta,a)} = \frac{1}{i} L_A(\zeta,a) T(\zeta),
\quad \quad \zeta\in\partial\Om, a\in\Om.
\end{equation}
A comparison with (\ref{S1L1}) and analytic continuation implies that the function $L^{(1)}(z,w)$ in Nehari's construction is the function $L_A(z,w)$. Thus,
\begin{equation}
L_{\varphi}(z,w) = \frac{L_A(z,w) q(z)}{q(w)}
\end{equation}
when $q(w)\neq 0$. In the case $q(w) = 0$, if $\eta$ is a point close to $w$, we have
\begin{equation}
L_{\varphi}(z,w) = \lim\limits_{\eta\rightarrow w} L_{\varphi}(z,\eta) 
=
\lim\limits_{\eta\rightarrow w}\frac{L_A(z,\eta) q(z)}{q(\eta)}.
\end{equation}
This finishes the proof.
\end{proof}

\begin{rem}\label{mrem}
\begin{enumerate}
    \item Theorem \ref{neh} holds if $q$ is taken to be any function in $A^{\infty}(\Om)$ such that $\vert q(z)\vert^2 = \varphi(z)$ on $\partial\Om$, without any restriction on the number of zeros. This is because the restriction on the number of zeros was nowhere used in the proof.

    \item It is worthwhile looking at Theorem 6 and Theorem 7 in \cite{Nehari}. In Theorem 6, the extremal function $F$ on the domain $D$ that maximizes $\vert f'(\zeta)\vert$ for $\zeta\in D$ where $f\in\mathcal{O}(D)$, $\vert f(z)\vert\leq 1$ for $z\in D$ and $f(\zeta) = f(\alpha_1) = \cdots=f(\alpha_m)=0$ $(\alpha_1,\ldots,\alpha_m\in D)$, is given by
	\begin{equation}
	F(z) = \frac{S_A(z,\zeta)}{L_A(z,\zeta)}
	\end{equation}
	for $A = \sum_{i=1}^{m} 1\cdot \alpha_i$. This can be directly checked using the definitions of $S_A$ and $L_A$.
\end{enumerate}
\end{rem}

\begin{cor}\label{SC2}
	Let $\Om$ be a bounded simply connected domain with $C^{\infty}$-smooth boundary. Let $A=\sum_{i=1}^{\sigma}m_i\cdot z_i $ be the effective divisor, where $z_i\in\Om$ and $m_i$ are positive integers. Let $p\in A^{\infty}(\Om)$ be a function such that $p(z_i) = 0$ with multiplicity $m_i$ for $i=1,\ldots,\sigma$, and $p$ does not vanish elsewhere. Let 
	\[
	\varphi(z) = \vert p(z)\vert^2, \quad z\in\partial\Om.
	\]
	There exists a nonvanishing function $q\in A^{\infty}(\Om)$ such that $\vert q(z)\vert^2 = \varphi(z)$ for $z\in\partial\Om$. Note that the function $q$ is unique modulo a multiplicative constant of modulus $1$. Then,
	\begin{equation}\label{SaLa}
	S_A(z,w) = \frac{p(z) \overline{p(w)}}{q(z)\overline{q(w)}} \,S(z,w)
	\quad
	\text{and}
	\quad
	L_A(z,w) = \frac{q(z)p(w)}{p(z)q(w)} \,L(z,w).
	\end{equation}
\end{cor}
\begin{proof}
	If $q_1$ and $q_2$ are two candidates for $q$, then $q_1/q_2$ is a non-vanishing holomorphic function in $A^{\infty}(\Om)$ with modulus $1$ on $\partial\Om$. This forces $q_1/q_2$ to be constant by applying the maximum modulus principle. Thus, $q$ is unique modulo a multiplicative constant of modulus $1$. By corollary \ref{SC1}, we have
    \[
    S_{\varphi}(z,w) = \frac{S(z,w)}{q(z)\overline{q(w)}} 
    \quad
    \text{and}
    \quad
    L_{\varphi}(z,w) = \frac{L(z,w) q(z)}{q(w)}.
    \]
    It also follows from Theorem \ref{neh} and Remark \ref{mrem} (1) that 
     \[
    S_{\varphi}(z,w) = \frac{S_A(z,w)}{p(z)\overline{p(w)}} 
    \quad
    \text{and}
    \quad
    L_{\varphi}(z,w) = \frac{L_A(z,w) p(z)}{p(w)}.
    \]
    Therefore, comparing the two above relations gives (\ref{SaLa}).
\end{proof}

\begin{exam}
	As a very simple example for Corollary \ref{SC2}, let $\Om = \mathbb{D}$ and $A=n\cdot 0$ be the effective divisor, where $n$ is a positive integer. 
    
    \medskip
    
    We can take $p(z) = z^n$. Now, $\varphi(z) = \vert p(z)\vert^2 \equiv 1$ for $z\in\partial\mathbb{D}$. Thus, we can take $q(z) \equiv 1$ to apply Corollary \ref{SC2}. We therefore have
	\[
	S_A(z,w) = z^n \ov{w}^n S_{\mathbb{D}}(z,w) = \frac{1}{2\pi} \frac{z^n\ov{w}^n}{1-z\ov{w}}
	\quad
	\text{and}
	\quad
	L_{A}(z,w) = \frac{w^n}{z^n} L_{\mathbb{D}}(z,w) = \frac{1}{2\pi} \frac{w^n}{z^n} \frac{1}{z-w}.
	\]
\end{exam}


	\section{Weighted Carath\'eodory metric and properties}
	
	The Carathéodory metric $c_{\Om}$ on a smoothly bounded domain $\Om\subset\mathbb{C}$ is defined by
	\begin{equation}
	c_{\Om}(z)=\sup\{\vert f'(z)\vert:\, f:\Om\rightarrow\mathbb{D}\,\text{ is holomorphic and }\,f(z)=0\},\quad z\in \Om.
	\end{equation}
	For $z\in \Om$, the solutions of the above extremal problem are of the form $e^{i\theta} f_z$, where $f_z$ denotes the Ahlfors map of $\Om$ based at point $z$. It is known that
	\[
	f_z(w)=\frac{S(w,z)}{L(w,z)},\quad w\in \Om.
	\]
	Since $L(\cdot,z)$ has a pole at $w=z$ with residue $1/2\pi$, we obtain
	\begin{eqnarray*}
	    {f_z}'(z)&=&\lim_{w\rightarrow z}\frac{f_z(w)-f_z(z)}{w-z}=\lim_{w\rightarrow z}\frac{f_z(w)}{w-z}=\lim_{w\rightarrow z}\frac{S(w,z)}{L(w,z) (w-z)}=\frac{\lim_{w\rightarrow z} S(w,z)}{1/2\pi}
        \\&=&2\pi S(z,z).
	\end{eqnarray*}	
	Therefore, 
	\begin{equation}
	c_{\Om}(z)=2\pi\, S(z,z).
	\end{equation}
	
	\begin{defn}[Weighted Carathéodory metric]
		Let $\Omega\subset\mathbb{C}$ be a bounded $n$-connected domain with $C^{\infty}$-smooth boundary such that no boundary curve is an isolated point. Let $\varphi$ be a positive real-valued $C^{\infty}$-smooth function on $\partial\Om$. We define the weighted Carathéodory metric $c_{\varphi}$ with respect to weight $\varphi$ by
		\begin{equation}
		c_{\varphi}(z) = 2\pi\, S_{\varphi}(z,z)
		\end{equation}
		for $z\in\Om$.
	\end{defn}
	
	There exists a function $q\in A^{\infty}(\Om)$ such that $\vert q(z)\vert^2 = \varphi(z)$ for $z\in\partial\Om$. Let $z_1,\ldots,z_{\sigma}$ be the zeros of $q$ with multiplicity $m_1,\ldots,m_{\sigma}$. Let 
    \[
    A=\sum_{i=1}^{\sigma}m_i\cdot z_i
    \]
    be the effective divisor corresponding to $q$. Then,
	\begin{equation}
	c_{\varphi}(z) = 2\pi S_{\varphi}(z,z) = 2\pi \frac{S_A(z,z)}{\vert q(z)\vert^2},\quad\quad z\neq z_{\mu}.
	\end{equation}

    It will be interesting to study further properties of the weighted Carathéodory metric. We leave it open for now.

	\section{Weighted Szeg\H{o} kernels in terms of the classical kernels}
	
	Let $\Omega\subset\mathbb{C}$ be a bounded $n$-connected domain with $C^{\infty}$-smooth boundary such that no boundary curve is an isolated point. Let $\gamma_j$, $j=1,\ldots,n$ denote the $n$ boundary curves of $\Omega$. The harmonic measure functions $\omega_j$ are unique harmonic functions on $\Omega$ that take value $1$ on $\gamma_j$ and $0$ on $\gamma_i$ for $i\neq j$. Let $F_j'$ denote the holomorphic functions on $\Omega$ given by $(1/2)(\partial/\partial z)\omega_j(z)$.
	
	\medskip
	
	Let $\varphi$ be a positive real-valued $C^{\infty}$-smooth function on $\partial\Om$. It is known (see \cite{Nehari_1}) that 
	\begin{equation}\label{1}
	S_{\varphi}(z,w) S_{1/\varphi}(z,w) = S(z,w)^2 + \sum_{\nu=1}^{n-1} a_{\nu}(w)F_{\nu}'(z)
	\end{equation}
	\begin{equation}\label{2}
	L_{\varphi}(z,w) L_{1/\varphi}(z,w) = L(z,w)^2 + \sum_{\nu=1}^{n-1} \overline{a_{\nu}(w)}F_{\nu}'(z) 
	\end{equation}
	where $a_{\nu}(w)$ are the coefficients depending on $w\in\Om$.
	
	\medskip
	
	It is known (see \cite{Be0, Bergman}) that if $K$ denotes the classical Bergman kernel of $\Om$ then
	\begin{equation}\label{ks}
	K(z,w) = 4\pi S(z,w)^2 + \sum_{j,k=1}^{n-1} c_{jk}F_{j}'(z) \overline{F_{k}'(w)}
	\end{equation}
	for some constants $c_{jk}$. In fact, if $a\in\Om$ is such that $S(\cdot,a)$ has $n-1$ distinct zeros $a_1,\ldots,a_{n-1}$ then (see \cite{Be0, Be5})
	\begin{equation}
	K(z,a) = 4\pi S(z,a)^2 + 2\pi\sum_{j=1}^{n-1} \frac{K(a_j,a)}{S'(a_j,a)}L(z,a_j)S(z,a).
	\end{equation}
	This is because the $n-1$ functions $L(z,a_j)S(z,a)$ form a basis for $\mathcal{F}'=span_{\mathbb{C}}\{ F_j'(z) : j=1,\ldots, n-1 \}$ (a theorem due to Schiffer, see \cite{Be0, Be5, Sch}). This result due to Schiffer also implies that
	\begin{equation}\label{ks2}
	K(z,w) = 4\pi S(z,w)^2 + \sum_{j,k=1}^{n-1} \lambda_{jk} L(z,a_j)S(z,a) \overline{L(w,a_k)S(w,a)}
	\end{equation}
	for some constants $\lambda_{jk}$ (see \cite{Be4} as well).
	
	\medskip
	
Coming back to (\ref{1}), we will see that
\begin{equation}\label{aid}
a_{\nu}(w) = \sum_{\mu=1}^{n-1} A_{\mu\nu}\overline{F_{\mu}'(w)}
\end{equation}
for some constants $A_{\mu\nu}$. This can be seen following an argument in \cite[see page 137]{Be0}. For $k=1,\ldots,n-1$
\[
\int_{\gamma_{k}} (S_{\varphi}(z,w) S_{1/\varphi}(z,w) - S(z,w)^2 ) \,dz
=
\sum_{\nu=1}^{n-1}a_{\nu}(w)\int_{\gamma_{k}} F_{\nu}'(z)\,dz 
=
\sum_{\nu=1}^{n-1} P_{k\nu} \,a_{\nu}(w)
\]
where $P=[P_{\mu\nu}]$ for $\mu,\nu=1,\ldots,n-1$, denotes the nonsingular matrix of periods 
\[
\int_{\gamma_{\mu}} F_{\nu}'(z) \,dz.
\]
Let $\widetilde{\gamma}_{k}$ be a curve that is homotopic to $\gamma_{k}$, but that is inside $\Om$. Then
\begin{equation}\label{3}
\int_{\widetilde{\gamma}_{k}} (S_{\varphi}(z,w) S_{1/\varphi}(z,w) - S(z,w)^2 ) \,dz
=
\sum_{\nu=1}^{n-1} P_{k\nu} \,a_{\nu}(w).
\end{equation}
Note that $a_{\nu}(w)$ are antiholomorphic functions in $w$. Now,
\begin{eqnarray*}
S_{\varphi}(z,w) S_{1/\varphi}(z,w) - S(z,w)^2 
&=& 
\overline{S_{\varphi}(w,z) S_{1/\varphi}(w,z) - S(w,z)^2}
= \overline{\sum_{\nu=1}^{n-1} a_{\nu}(z) F_{\nu}'(w)} 
\\&=& \sum_{\nu=1}^{n-1} \overline{a_{\nu}(z)} \overline{F_{\nu}'(w)}.
\end{eqnarray*}
The integral in (\ref{3}) can be approximated by a finite Riemann sum which will therefore lie in the linear span of functions $\{\overline{F_{\nu}'(w)}\}_{\nu=1}^{n-1}$. A limiting argument will show that the integral also lies in the linear span of $\overline{F_{\nu}'(w)}$. Thus, we have
\[
\sum_{\nu=1}^{n-1} Q_{k\nu} \overline{F_{\nu}'(w)} = 
\sum_{\nu=1}^{n-1} P_{k\nu} \,a_{\nu}(w),
\quad \quad k=1,\ldots,n-1
\]
for some constants $Q_{k\nu}$. Let $Q$ denote the matrix $[Q_{\mu\nu}]$. Therefore,
\[
P \,\cdot\,[a_{\nu}(w)]_{\nu=1}^{n-1}  = Q \,\cdot\,[\overline{F_{\nu}'(w)}]_{\nu=1}^{n-1}.
\]
Since $P$ is nonsingular, we obtain (\ref{aid}).
Following the same line of reasoning for (\ref{2}), we have
\begin{equation}\label{Neq4}
	S_{\varphi}(z,w) S_{1/\varphi}(z,w) = S(z,w)^2 + \sum_{j,k=1}^{n-1} a_{jk}\,F_{j}'(z)\overline{F_k'(w)}
\end{equation}
\begin{equation}\label{Neq5}
	L_{\varphi}(z,w) L_{1/\varphi}(z,w) = L(z,w)^2 + \sum_{j,k=1}^{n-1} \overline{a_{jk}}\, F_{j}'(z) F_{k}'(w) 
\end{equation}
for some constants $a_{jk}$. This also implies that
\begin{equation}\label{Neq6}
	S_{\varphi}(z,w) S_{1/\varphi}(z,w) = S(z,w)^2 + \sum_{j,k=1}^{n-1} \alpha_{jk}\,L(z,a_j)S(z,a) \overline{L(w,a_k)S(w,a)}
\end{equation}
\begin{equation}\label{Neq7}
	L_{\varphi}(z,w) L_{1/\varphi}(z,w) = L(z,w)^2 + \sum_{j,k=1}^{n-1} \overline{\alpha_{jk}}\, L(z,a_j)S(z,a) L(w,a_k)S(w,a)
\end{equation}
for some constants $\alpha_{jk}$.
Relations (\ref{Neq1}), (\ref{Neq2}), Theorem \ref{neh}, relations (\ref{Neq4}) -- (\ref{Neq7}) together may help to shed more light on the properties of the weighted Szeg\H{o} and weighted Garabedian kernels.

\medskip

Finally, we also observe from the relations (\ref{ks}), (\ref{ks2}), (\ref{Neq4}) and (\ref{Neq6}) that
\begin{equation}
4\pi\,S_{\varphi}(z,w) S_{1/\varphi}(z,w) = K(z,w) + \sum_{j,k=1}^{n-1} b_{jk}\,F_{j}'(z)\overline{F_k'(w)}
\end{equation}
\begin{equation}
4\pi\, S_{\varphi}(z,w) S_{1/\varphi}(z,w) = K(z,w) + \sum_{j,k=1}^{n-1} \beta_{jk}\,L(z,a_j)S(z,a) \overline{L(w,a_k)S(w,a)}
\end{equation}
for some constants $b_{jk}$ and $\beta_{jk}$.

\medskip

Thus, the weighted Szeg\H{o} kernels corresponding to weights $\varphi$ and $1/\varphi$ can be expressed in terms of the classical kernel functions.


\medskip

\textbf{Acknowledgment.} The work was commenced and developed while the first named author was an Institute Postdoctoral Fellow at the Department of Mathematics, Indian Institute of Technology, Kanjikode, Pudussery west, Palakkad, Kerala 678623, India, from 7 June 2024 to 6 December 2024. It was subsequently continued and finalized while the author was a Visiting Assistant Professor at the Institute of Mathematics, Jagiellonian University, ul. Lojasiewicza 6, 30-348 Kraków, Poland, from 15 Feb 2025 to 14 Feb 2026, supported by SHENG-3 K/NCN/000286, research project UMO2023/Q/ST1/00048.

\medskip

The authors sincerely thank the anonymous referee for a careful and thorough reading of the manuscript and for providing useful comments and suggestions. This has significantly improved the clarity and exposition of the article.

\medskip

\textbf{AI Declaration.}
No AI tools were used in developing or verifying the mathematical results and proofs, or in writing this manuscript.


\end{document}